\documentclass[12pt]{amsart}
\usepackage{amsmath}
\usepackage{amsfonts}
\usepackage[T1]{fontenc} %needed for correct accents
\usepackage{color}
\usepackage{url}

\usepackage{a4wide}
\usepackage{verbatim}
\usepackage[parfill]{parskip}    % Activate to begin paragraphs with an empty line rather than an indent
\usepackage{graphicx}
\usepackage{amssymb}
\usepackage{epstopdf}
\usepackage{psfrag}
\newcommand {\bea}{\begin{align}}
\newcommand {\ea}{\end{align}}
\newtheorem{proposition}{Proposition}[section]
\newtheorem{theorem}{Theorem}[section]

\newtheorem{lemma}{Lemma}[section]

\newtheorem{remark}{Remark}[section]

\newtheorem{corollary}{Corollary}[section]

\usepackage{xspace}

\usepackage{subfigure}

\newcommand{\thmref}[1]{{Theorem~\ref{#1}}}
\newcommand{\lemref}[1]{{Lemma~\ref{#1}}}
\newcommand{\secref}[1]{{Section~\ref{#1}}}

\newcommand{\propref}[1]{{Proposition~\ref{#1}}}
\newcommand{\coref}[1]{{Corollary~\ref{#1}}}

\newcommand{\eps}{\epsilon}

\begin{document}
\title[Sharp interface limit of the stochastic Cahn-Hilliard equation]
{Improved estimates for the sharp interface limit of the stochastic Cahn-Hilliard equation with space-time white noise}

\author{\v{L}ubom\'{i}r Ba\v{n}as}
\address{Department of Mathematics, Bielefeld University, 33501 Bielefeld, Germany}
\email{banas@math.uni-bielefeld.de}
\author{Jean Daniel Mukam}
\address{Department of Mathematics, Bielefeld University, 33501 Bielefeld, Germany}
\email{jmukam@math.uni-bielefeld.de}

\begin{abstract}
We study the sharp interface limit of the stochastic Cahn-Hilliard equation with cubic double-well potential and additive space-time white noise
$\epsilon^{\sigma}\dot{W}$, where $\epsilon>0$ is an interfacial width parameter.   
We prove that, for sufficiently large scaling constant $\sigma >0$,
the stochastic Cahn-Hilliard equation converges to the deterministic Mullins-Sekerka/Hele-Shaw problem for $\epsilon\rightarrow 0$. 
The convergence is shown in suitable fractional Sobolev norms as well as in the $L^p$-norm for $p\in (2, 4]$ in spatial dimension $d=2,3$.
This generalizes the existing result for the space-time white noise to dimension $d=3$  and improves the existing results  for smooth noise, which were so far limited to $p\in \left(2,\frac{2d+8}{d+2}\right]$ in spatial dimension $d=2,3$.
%As a byproduct of the analysis of the stochastic problem with the space-time white noise we also obtain improved estimates for the sharp interface limit 
%{\red of the stochastic Cahn-Hilliard equation with a more regular noise with borderline regularity that still admits $\mathbb{H}^1$ spatial regularity of the solution}
%and of the deterministic Cahn-Hilliard equation.
As a byproduct of the analysis of the stochastic problem with space-time white noise, we identify minimal regularity requirements on the noise which allow convergence to the sharp interface limit in the $\mathbb{H}^1$-norm
and also provide improved convergence estimates for the sharp interface limit of the deterministic problem.
\end{abstract}

\maketitle

%\begin{keyword}
%Stochastic Cahn-Hilliard equation  \sep Hele-Shaw/Mullins-Sekerka problem \sep Sharp interface limit \sep Space-time white noise \sep Errors estimate.

%\end{keyword}

\section{Introduction}
\label{Introduction}
We consider the stochastic Cahn-Hilliard equation with  additive space-time white noise
%\begin{subequations}
\begin{align}
\label{sch} %\label{model1}
du^{\epsilon} & =\Delta\left(-\epsilon\Delta u^{\epsilon}+\frac{1}{\epsilon}f(u^{\epsilon})\right)dt+\epsilon^{\sigma}dW(t)
\quad && \text{in } \mathcal{D}_T:=(0, T)\times\mathcal{D},
\\ 
%\label{Boundary}
\nonumber
%\frac{\partial u^{\epsilon}}{\partial \vec{n}} & =\frac{\partial\Delta u^{\epsilon}}{\partial \vec{n}}=0  && \text{on } (0, T)\times \partial \mathcal{D},
\partial_{\vec{n}} u^{\epsilon} & = \partial_{\vec{n}} \Delta u^{\epsilon}=0  && \text{on } (0, T)\times \partial \mathcal{D},
\\
\nonumber
u^{\epsilon}(0) & = u^{\epsilon}_0  && \text{in } \mathcal{D},
\end{align}
%\end{subequations}
where $\sigma>0$, $T>0$ are fixed constants and $\epsilon>0$ is a (small) interfacial width parameter.
For simplicity we take $\mathcal{D}=(0, 1)^d$ to be the unit cube in $\mathbb{R}^d$, $d=2,3$,
with $\vec{n}$ the outer unit normal to $\partial\mathcal{D}$ and $W$ is the space-time white noise.
The nonlinearity $f$ in (\ref{sch}) ensures that asymptotically the solutions $\{u^{\eps}\}_{\eps>0}$ of (\ref{sch}) remain within the physically meaningful range $-1 \leq u^{\epsilon} \leq 1$.
One of the most widely used choices is $f(u) = F'(u) = u^3-u$
where $F(u)=\frac{1}{4}(u^2-1)^2$ is a double-well potential with minima at $\pm 1$. 

Formally, equation \eqref{sch} can be equivalently written in the mixed form
\begin{equation}
\label{model2}
\begin{array}{rclll}
du^{\epsilon} & = & \Delta w^{\epsilon}dt+\epsilon^{\sigma}dW(t)\quad & \text{in}\; \mathcal{D}_T,\\
w^{\epsilon} & = &-\epsilon\Delta u^{\epsilon}dt+\frac{1}{\epsilon}f(u^{\epsilon})\quad & \text{in}\; \mathcal{D}_T,
\end{array}
\end{equation}
where $w^{\epsilon}$ is the so-called chemical potential. For sufficiently smooth noise and data the
chemical potential $w^{\epsilon}$ has sufficient regularity so that the formulation \eqref{model2} can be made rigorous, cf. \cite{Banas19}.

The Cahn-Hilliard equation is a prototype  model for mass conservative phase separation and
coarsening phenomena in binary alloys \cite{Cahn1,Cahn2}.   
The solution $u^{\epsilon}$ of \eqref{sch} is an order parameter 
which approaches $\pm1$ in the regions occupied by the pure phases. The pure phases are separated by a 
thin layer where $|u^{\epsilon}| < 1$, so-called diffuse interface, 
with thickness proportional to the (small) interfacial width parameter $\epsilon$.
It has been observed in \cite{GAMEIRO2005693} that 
the simulation results obtained with the stochastic Cahn-Hilliard equation with the space-time white noise  
are in better agreement with physical experiments than those obtained by deterministic simulations.

In the deterministic setting (i.e., when $W\equiv 0$ in (\ref{sch})) the Cahn-Hilliard equation reads as 
\begin{align}
\label{model3}
\begin{array}{rcllll}
\partial_t u^{\epsilon}_D & = &\Delta w^{\epsilon}_D\quad & \text{ in } \mathcal{D}_T,\\
w^{\epsilon}_D & = &-\epsilon\Delta u^{\epsilon}_D+\frac{1}{\epsilon}f(u^{\epsilon}_D) \quad &\text{ in } \mathcal{D}_T.
\end{array}
\end{align}
The sharp interface limit of the deterministic Cahn-Hilliard equation has been analyzed in \cite{abc94}, where
it was shown that, for $\epsilon\rightarrow 0$, the function $w^{\epsilon}_D$ tends to a function $v$, which together with the free boundary $\Gamma_t:=\lim_{\epsilon\rightarrow0}\Gamma_t^{\epsilon}$
(where $\Gamma_t^{\epsilon}:=\{x\in \mathcal{D}:\; {u^{\epsilon}_D(t,x)=0}\}$, $t\in (0, T)$) satisfies the Mullins-Sekerka/Hele-Shaw problem:
\begin{align}
\label{model4}
\begin{array}{lllll}
 \Delta v & =0\quad & \text{in}\; \mathcal{D}\setminus \Gamma_t,\\
 \partial_{\vec{n}}v & =0\quad & \text{on}\; \partial\mathcal{D},\\
 v & =\lambda H\quad & \text{on}\; \Gamma_t,\\
 \mathcal{V} & =\frac{1}{2}(\partial_{\vec{n}_{\Gamma}}v^{+}-\partial_{\vec{n}_{\Gamma}}v^{-})\quad & \text{on}\; \Gamma_t,\\
 \Gamma_0 & =\Gamma_{00},
\end{array}
\end{align}
where $H$ is the mean curvature of $\Gamma_t$, $\mathcal{V}$ is the normal velocity of the interface, 
$\vec{n}_{\Gamma}$ is the unit normal vector to $\Gamma_t$ and $v^{+}$, $v^{-}$ are respectively the restriction 
of $v$ on $\mathcal{D}^{+}_t$, $\mathcal{D}^{-}_t$ (the exterior and interior of $\Gamma_t$ in $\mathcal{D}$). % $\lambda=\frac{1}{2}\int_{-1}^1F(s)ds$

In the stochastic setting  with trace class noise, it is shown in \cite{abk18} that, for suitable scaling of the noise, the sharp interface limit of the Cahn-Hilliard equation is the deterministic Hele-Shaw model (\ref{model4}) (cf. \cite{Antonopoulou_2023} for the case of multiplicative noise). The work \cite{Banas19} studies convergence of the numerical approximation of the stochastic Cahn-Hilliard equation
with smooth noise to the sharp interface limit and also obtains the first result on uniform pointwise convergence to the deterministic Hele-Shaw problem (\ref{model4}).
The sharp interface limit (\ref{model4}) of the Cahn-Hilliard equation with singular space-time white noise has been studied
in \cite{BYZ22}. We note that all aforementioned results (also including the present work) 
rely on the spectral estimate for the deterministic problem, cf. \cite{abc94}, 
the use of which requires appropriate scaling of the noise with respect to the interfacial width parameter $\eps$.
 We also mention the recent work \cite{sch_aposter}
which employs a (discrete) stochastic counterpart of the principal eigenvalue problem to 
derive a posteriori error estimates for the numerical approximation of the stochastic Cahn-Hilliard equation
 and \cite{yang2019} where a stochastic Hele-Shaw problem is obtained as the sharp-interface limit of the stochastic Cahn-Hilliard equation with (unscaled) smooth-in-time noise.

Throughout this paper we assume that for a given smooth closed hypersurface $\Gamma_{00}\subset \mathcal{D}$
the Hele-Shaw problem \eqref{model4} admits a smooth solution $(v, \{\Gamma_t\}_{t\in[0, T]})$.
Under this assumption, it is possible to construct an approximation $(u^{\epsilon}_{\mathsf{A}}, w^{\epsilon}_{\mathsf{A}})$ of \eqref{model3} that satisfies 
\begin{align}
\label{model5}
\begin{array}{rclll}
\partial_t u^{\epsilon}_{\mathsf{A}} & = &\Delta w^{\epsilon}_{\mathsf{A}}\quad  &\text{in}\; \mathcal{D}_T,\\
w^{\epsilon}_{\mathsf{A}} & = & -\epsilon\Delta u^{\epsilon}_{\mathsf{A}}+\frac{1}{\epsilon}f(u^{\epsilon}_{\mathsf{A}})+r^{\epsilon}_{\mathsf{A}}\quad & \text{in}\; \mathcal{D}_T,
%\\
%\partial_{\vec{n}} u^{\epsilon}_{\mathsf{A}} & ={\partial_{\vec{n}}\Delta w^{\epsilon}_{\mathsf{A}}}=0\quad & \text{on}\; \partial\mathcal{D}.
\end{array}
\end{align}
with the same boundary conditions as in \eqref{sch},  cf., \cite[(4.30)]{abc94}.
Furthermore, for any $K>0$ and
\begin{align*}
k>(d+2)\frac{d^2+6d+10}{4d+16},
\end{align*} 
the following estimates hold, cf.  \cite[Theorems 2.1 \& 4.12]{abc94}  and \cite[(4.30)]{abc94}:
\begin{subequations}
\label{Alakiresult1}
\begin{align}
\label{contconv}
&\Vert r^{\epsilon}_{\mathsf{A}}\Vert_{C(\mathcal{D}_T)} \leq C\epsilon^{K-2},\quad \Vert w^{\epsilon}_{\mathsf{A}}-v\Vert_{C(\mathcal{D}_T)}  \leq C\epsilon,
\\ \label{lpconv}
&\quad \Vert u^{\epsilon}_{D}-u^{\epsilon}_{\mathsf{A}}\Vert_{ L^p(0, T; \mathbb{L}^p)} \leq C\epsilon^k\quad \text{ for }  p\in\left(2, \frac{2d+8}{d+2}\right],
\end{align} 
\end{subequations}
where the constant $C\geq 0$ is independent of $\epsilon$.

For $d=2$ the best possible space in (\ref{lpconv}) is  $L^3(0, T; \mathbb{L}^3)$, for $d=3$ the best space in (\ref{lpconv}) is  $L^{\frac{14}{5}}(0, T; \mathbb{L}^{\frac{14}{5}})$.
This convergence result is suboptimal in the case of the double-well potential where  $u^\eps, u^{\epsilon}_{\mathsf{A}}\in L^4(0, T; \mathbb{L}^4)$.
%Moreover, for $d=3$ the best space in (\ref{lpconv}) is $L^{\frac{14}{5}}(\mathcal{D}_T)$ which excludes the case of the double-well potential (i.e., the cubic nonlinearity $f$ in (\ref{model3})).
%However,  $u^\eps, u^{\epsilon}_{\mathsf{A}}\in L^4(\mathcal{D}_T)$; so it would be optimal to establish convergence results in $L^4(\mathcal{D}_T)$. 

%Weak convergence of the deterministic Cahn-Hilliard equation was also investigated in \cite{Chen96}. 
%{\red Results on the sharp-interface limit of the Cahn-Hilliard equation in the stochastic setting are relatively recent.}

%Recently the sharp interface limit of the stochastic problem (\ref{sch}) with additive trace-class noise was analyzed in \cite{abk18}.  
In the stochastic setting with trace-class noise one has the following error estimates, cf., \cite[Theorem 3.10]{abk18}:
\begin{align}
\label{Antonopoulouresult1}
\mathbb{P}\left(\left\{\Vert u^{\epsilon}-u^{\epsilon}_{\mathsf{A}}\Vert_{ L^p(0, T; \mathbb{L}^p)}\leq C\epsilon^{\gamma}\right\}\right)\geq 1-C_l\epsilon^l,\\
\label{Antonopoulouresult2}
\mathbb{P}\left(\left\{\Vert u^{\epsilon}-u^{\epsilon}_{\mathsf{A}}\Vert^2_{L^{\infty}(0, T; \mathbb{H}^{-1})}\leq C\epsilon^{g(\sigma, \gamma)}\right\}\right)\geq 1-C_l\epsilon^l,\\
\label{Antonopoulouresult3}
\mathbb{P}\left(\left\{\Vert u^{\epsilon}-u^{\epsilon}_{\mathsf{A}}\Vert^2_{L^2(0, T; \mathbb{H}^1)}\leq C\epsilon^{h(\sigma, \gamma)}\right\}\right)\geq 1-C_l\epsilon^l,
\end{align}
for suitable $\gamma>0$, $l>0$, where $u^{\epsilon}_{\mathsf{A}}$ is the (deterministic) solution of (\ref{model5}) and
$p$ is as in \eqref{lpconv}.  
Hence,  as in the deterministic setting, the best spaces in which the convergence for the stochastic Cahn-Hilliard equation takes place in $d=2$ and $d=3$
are  $L^3(0, T; \mathbb{L}^3)$ and  $L^{\frac{14}{5}}(0, T; \mathbb{L}^{\frac{14}{5}})$, respectively.

The sharp interface limit of the two dimensional stochastic Cahn-Hilliard equation driven by singular space-time white noise was recently analyzed 
in \cite{BYZ22} where error estimate \eqref{Antonopoulouresult1} (with $p=3$) and \eqref{Antonopoulouresult2} were obtained. 
Due to regularity restrictions, an analogue of the error estimate \eqref{Antonopoulouresult3} 
is not available in the case of  space-time white noise.

In the previous works on the sharp interface limits of the deterministic and the stochastic Cahn-Hilliard equation \cite{abc94,abk18}, the following inequality (see \cite[Lemma 2.2]{abc94})
was employed to estimate the nonlinearity:
\begin{align}\label{xxx1}
-\int_{\mathcal{D}}\epsilon^{-1}\mathcal{N}(u^{\epsilon}_{\mathsf{A}},v)v\leq C\epsilon^{-1}\Vert v\Vert^p_{\mathbb{L}^p},\quad p\in(2,3],
\end{align}
where $\mathcal{N}(u,v):=f(u+v)-f(u)-f'(u)v$. 
The above estimate is combined with dimension dependent interpolation inequalities
which yield suboptimal dimension dependent estimates for the double-well  nonlinearity. 
In the stochastic case with the space-time white noise, 
an analogous approach restricts the analysis to spatial dimension $d=2$, cf. \cite{BYZ22}.

%{\red This therefore restricts the space in which  convergence analyses are studied to $L^3(\mathcal{D}_T)$ as the largest possible one.  
%To handle cubic nonlinearities in dimension $d=3$, one needs a further restriction, namely $p\leq \frac{2d+8}{d+2}=\frac{14}{5}$, see e.g., \cite{abk18,abc94}. }

The (probabilistically) strong variational solution of the stochastic Cahn-Hilliard equation enjoys the following regularity : (i) $u\in L^4\left(\Omega; C([0, T], \mathbb{L}^4)\right)\cap L^2\left(\Omega; L^2(0, T; \mathbb{H}^1)\right)$ for 
{sufficiently regular} trace class noise (see \cite[Proposition 2.2]{Debussche1}) and (ii) $u\in L^4\left(\Omega; C([0, T], \mathbb{L}^4)\right)\cap L^2\left(\Omega; L^2(0, T; \mathbb{H}^{2-\frac{d}{2}-\vartheta})\right)$
for {arbitrary} $\vartheta>0$ for  space-time white noise (see \thmref{regularitytheorem} below).
In the present work, 
instead of employing the general formula (\ref{xxx1}), we estimate the double-well nonlinearity by an explicit calculation
and use a new interpolation inequality (\lemref{Fundamentallemma} below). This approach yields
error estimates which are optimal with respect to the aforementioned regularity of the solution of the (stochastic) Cahn-Hilliard equation and also allows us to generalize the analysis to the case 
of the space-time white noise in dimension $d=3$.
%Note that even for $d=2$ the error estimates in \cite{abc94,abk18,BYZ22} are only in $L^3(\mathcal{D}_T)$, while here we have error estimates in $L^4(\mathcal{D}_T)$.

%\subsection{Outline of our results}
The main contributions of the present paper are the following.
\begin{itemize}
\item[(i)] We prove \eqref{Alakiresult1} for $p\in(2,4]$ for any $d=2,3$, see \thmref{mainresult2}. {This improves \cite[Theorem 2.1]{abc94} for the double-well potential}.
\item[(ii)] We prove \eqref{Antonopoulouresult1} and \eqref{Antonopoulouresult2} for stochastic Cahn-Hilliard equation driven by space-time white noise in dimension $d=2,3$ with $p\in(2,4]$, see \thmref{mainresult1}. 
%Note that this differs from the results in \cite{BYZ22} in the following two points.
%\begin{itemize}
%\item[(a)] we obtained the convergence result in the largest space $L^4(\mathcal{D}_T)$.
%\item[(b)] Only the case $d=2$ was  investigated in \cite{BYZ22}, while here we consider $d=2,3$. Note that  some essential Sobolev embeddings useful to handle the cubic term valid in dimension $d=2$ are no longer valid in dimension $d=3$.
%\end{itemize} 
\item[(iii)] We derive an analogue of the error estimate \eqref{Antonopoulouresult3} (see \thmref{mainresult1}) in fractional Sobolev spaces
\begin{align*}
%\label{LuboDaniel1}
\mathbb{P}\left(\left\{\Vert u^{\epsilon}-u^{\epsilon}_{\mathsf{A}}\Vert^2_{L^2(0, T; \mathbb{H}^{2-\frac{d}{2}-\vartheta})}\leq C\epsilon^{\frac{\gamma}{3}  -4}\right\}\right)\geq 1-C_{\delta,\eta}\epsilon^{\delta+\eta}-C_{\vartheta,\kappa}\epsilon^{\vartheta+\kappa}, 
\end{align*}
for any $\vartheta,\delta>0$ and $\kappa, \eta\geq 0$. We observe that for $d=2$, this leads to an error estimate almost in $\mathbb{H}^1$ and for $d=3$ this leads to an error almost in $\mathbb{H}^{\frac{1}{2}}$.  
Note that it is not clear whether an error estimate in $L^2(0, T; \mathbb{H}^1)$ is achievable in the low regularity setting of  space-time white noise.
%{ \red \item[(iv)]  (remove?) We derive an error estimate for the difference between the solutions of the stochastic and deterministic Cahn-Hilliard equations in norms that are compatible with the regularity of the solutions, see Corollary. }
%{ \red \item[(iv)]  (UPDATE) In section 5+6 we obtain improvements in the case of more regular noise and in the deterministic case. }
 \item[(iv)] We identify minimal regularity properties of the noise required for the $\mathbb{H}^1$ convergence \eqref{Antonopoulouresult3} to hold, see \secref{sec_h1}. The condition is weaker than the one required in \cite[Assumption 3.1]{abk18}.  
\end{itemize}      

We adopt the approach of \cite{Debussche1}, \cite{BYZ22} which is based on introducing the stochastic convolution \eqref{Convolution}
and studying the translated solution $Y^\eps := u^{\epsilon}-u^{\epsilon}_{\mathsf{A}}-Z^{\epsilon}$. 
We derive optimal a priori estimate for the translated solution $Y^\eps$ in \lemref{Formulalemma}. 
{In order to obtain estimates that are robust with respect to the interfacial width parameter $\eps$ (i.e., to avoid the use of Gronwall's lemma) 
we employ the lower bound of the principal eigenvalue of the linearized (deterministic) Cahn-Hilliard equation, cf. \cite[Theorem 3.1]{abc94}}.  
In order to overcome the barrier $p\leq \frac{2d+8}{d+2}$ 
we make use of a new interpolation inequality (see \lemref{Fundamentallemma}). 
%{\red This interpolation inequality is also an important ingredient to handle the higher dimensional case $d=3$ for the stochastic Cahn-Hilliard equation with space-time white noise.}
To deal with the low regularity of the space-time white noise
we benefit from  the smoothing properties of the semigroup generated by the bi-Laplacian $\Delta^2$. Consequently we obtain an estimate for the convolution  \eqref{Convolution}
in fractional Sobolev spaces in \lemref{bruitlemma1}, which allows us to derive the error estimate \eqref{result5} below.
%{which allow us to build subsets of high probability $\widetilde{\Omega}_{\vartheta,\kappa,\epsilon}\subset\Omega$ (see \eqref{defOmega2} for a precise definition), see \lemref{Semigroup}.}

%\subsection{Outline of the paper}
 The paper is organized as follows.
 We introduce the notation and preliminary results in  \secref{sec_notation}.
Some useful regularity properties of the variational solution of \eqref{sch} are presented in 
\secref{sec_exist}. 
The sharp interface limit of the stochastic Cahn-Hilliard equation is analyzed in \secref{sec_stoch}.
The corresponding results for more regular noise are summarized in \secref{sec_h1} and the deterministic problem is analyzed in \secref{sec_det}.

\section{Notations and preliminaries}
\label{sec_notation}
By $\mathbb{L}^p:=L^p(\mathcal{D})$ we denote the standard Lebesgue space of $p$-th order integrable functions
on $\mathcal{D}$. The $\mathbb{L}^2$ inner product is denoted as $(.,.)$ and the associated norm as $\Vert.\Vert$. 
For $g\in \mathbb{L}^2$, we denote by $m(g)$ the average of $g$, given by $m(g):=\frac{1}{\vert \mathcal{D}\vert}\int_{\mathcal{D}}g(x)dx$. We denote by  $\mathbb{L}^2_0:=\{g\in \mathbb{L}^2:\; m(g)=0\}$. 
For $s\in\mathbb{R}$ we denote the standard Sobolev space on $\mathcal{D}$ by $\mathbb{H}^s:=H^s(\mathcal{D})$.

The Neumann Laplace operator $-\Delta$ with domain $ D(-\Delta)=\{u\in \mathbb{H}^2:\; \frac{\partial u}{\partial \vec{n}}=0\;\text{on}\; \partial\mathcal{D}\}$ 
is self-adjoint, positive and has compact resolvent. 
We consider an orthonormal basis of $\mathbb{L}^2$ 
consisting of eigenvectors $\{e_j\}_{j\in\mathbb{N}^d}$ of the Neumann Laplacian, with corresponding eigenvalues $(\lambda_j)$ such that
$0=\lambda_0 < \lambda_1\leq \lambda_2\leq \cdots \lambda_j\longrightarrow +\infty$.
Note that for $k=(k_1,\cdots, k_d)\in \mathbb{N}^d$, $\lambda_k$ satisfies $\lambda_k\simeq \vert k\vert^2$, where $\vert k\vert^2=\lambda_1^2+\cdots+\lambda_d^2$  and 
%{\magenta for $\alpha<-\frac{d}{2}$} 
it holds that
\begin{align}
\label{convergencecondition}
\sum_{j\in \mathbb{N}^d}\lambda_j^{\alpha}<+\infty  \quad \text{iff} \quad \alpha<-\frac{d}{2}.
\end{align} 

For $s\in \mathbb{R}$ and $u\in \mathbb{L}^2$ we define the fractional Laplacian $(-\Delta)^s$ as
\begin{align}
\label{Fractionaire1}
(-\Delta)^su=\sum_{j\in \mathbb{N}^d}\lambda^{s}_ju_je_j\quad \text{for}\; u=\sum_{j\in \mathbb{N}^d}u_je_j,
\end{align}
where the domain of $(-\Delta)^{\frac{s}{2}}$ is given by
\begin{align*}
D((-\Delta)^{\frac{s}{2}}):=\left\{u=\sum_{j\in\mathbb{N}^d}u_je_j:\;\sum_{j\in\mathbb{N}^d}\lambda_j^su_j^2<\infty\right\}.
\end{align*}
We introduce the seminorm and semiscalar product
\begin{align*}
\vert v\vert_s=\Vert (-\Delta)^{\frac{s}{2}}v\Vert\quad \text{and}\quad ( u, v)_s=\left( (-\Delta)^{\frac{s}{2}}u, (-\Delta)^{\frac{s}{2}}v\right)
\quad u, v\in D((-\Delta)^{\frac{s}{2}})
\end{align*} 
as well as the  norm
\begin{align*}
\Vert v\Vert_s=\left(\vert v\vert^2_s+m^2(v)\right)^{\frac{1}{2}}\quad v\in D((-\Delta)^{\frac{s}{2}}). 
\end{align*}
For $s\in[0, 2]$ the norm $\Vert.\Vert_s$ is equivalent to the usual norm on $\mathbb{H}^s$ and $D((-\Delta)^{\frac{s}{2}})$ is a closed subspace of $\mathbb{H}^s$,
see, e.g., \cite[Section 2.1]{Debussche1}.
% We denote by $\mathcal{L}^0_2$ the space of Hilbert-Schmidt operators from $\mathbb{L}^2$ to itself, equipped with the norm 
%\begin{align}
%\label{HilbertSchmidtnorm}
%\Vert L\Vert_{\mathcal{L}^0_2}=\left(\sum_{i=1}^{\infty}\Vert Le_i\Vert^2\right)^{\frac{1}{2}},\quad L\in \mathcal{L}^0_2.
%\end{align}
% Note that \eqref{HilbertSchmidtnorm} is independent of the choice of the orthonormal basis of $\mathbb{L}^2$. 
 
The term  $W$ in \eqref{sch} is the space-time white noise, which is formally represented as
\begin{align}
\label{SpaceTimeWhiteNoise}
W(x,t)=\sum_{j\in\mathbb{N}^d}\beta_j(t)e_j(x),
\end{align}
where $\beta_j$ are independent and identically distributed standard Brownian motions on a filtered probability space $(\Omega, \mathcal{F}, \{\mathcal{F}_t\}_t, \mathbb{P})$,
cf. \cite{Debussche1}. 
{Note that the space-time white noise enjoys the zero-mean property, i.e., it holds that $m(W) = 0$.}

We make use of the following spectral estimate of the deterministic problem, cf. \cite[Proposition 3.1]{abc94}.
 \begin{proposition}
 \label{Spectralestimate}
 Let $u^{\epsilon}_{\mathsf{A}}$ be the approximation in \eqref{model5}. Then for all $w\in \mathbb{H}^1$  with $\int_{\mathcal{D}}wdx=0$, the following   holds
 \begin{align*}
 \int_{\mathcal{D}}\left(\epsilon\vert\nabla w\vert^2+\frac{1}{\epsilon}f'(u^{\epsilon}_{\mathsf{A}})w^2\right)\geq -C_0\Vert w\Vert^2_{\mathbb{H}^{-1}},
 \end{align*}
 where $C_0\geq 0$ is a constant independent of $w$ and $\epsilon$. 
 \end{proposition}

\section{Existence and regularity of the solution}\label{sec_exist}
In this section we summarize existence and regularity properties of the solution of the stochastic Cahn-Hilliard  equation (\ref{sch})
with  space-time white noise.

We introduce the stochastic convolution
\begin{equation}
\label{Convolution}
Z^{\epsilon}(t):=\epsilon^{\sigma}\int_0^te^{-(t-s)\epsilon\Delta^2}dW(s)=\epsilon^{\sigma}\sum_{i\in\mathbb{N}^d}\int_0^te^{-\lambda_i^2(t-s)\epsilon}e_id\beta_i(s)\quad t\in[0, T].
\end{equation}
 The following two lemmas will be useful in the rest of this paper.
\begin{lemma}
\label{convollemma}
For any $p\in[1, \infty)$, there exists a  constant $C=C(p)\geq 0$ such that
\begin{align*}
\mathbb{E}\left[\sup_{t\in[0, T]} \Vert Z^{\epsilon}(t)\Vert^p_{\mathbb{L}^p}\right]\leq C(p)\epsilon^{(\sigma-\frac{1}{2})p}.
\end{align*}
\end{lemma}

\begin{proof}
Taking the expectation in both sides of \eqref{Convolution} yields
\begin{align*}
\mathbb{E}\left[\sup_{t\in[0, T]}\Vert Z^{\epsilon}(t)\Vert^p_{\mathbb{L}^p}\right]=\epsilon^{\sigma p}\int_{\mathcal{D}}\mathbb{E}\left[\sup_{t\in[0, T]}\left\vert\sum_{i\in\mathbb{N}^d}\int_0^te^{-\lambda_i^2(t-s)\epsilon}d\beta_i(s)e_i(x)\right\vert^p\right]dx.
\end{align*}
Using the Burkholder-Davis-Gundy (BDG) inequality  \cite[Theorem 4.36]{DaPratoZabczyk} and the uniform boundedness of $(e_i)_{i\in \mathbb{N}}$, we obtain
\begin{align*}
\mathbb{E}\left[\sup_{t\in[0, T]}\Vert Z^{\epsilon}(t)\Vert^p_{\mathbb{L}^p}\right]
%&\leq& \epsilon^{\sigma p}\int_{\mathcal{D}}\left(\mathbb{E}\vert\sum_{i\in\mathbb{N}^d}\int_0^Te^{-\lambda_i^2(T-s)\epsilon}d\beta_i(s)e_i(x)\vert^2\right)^{\frac{p}{2}} dx\nonumber\\
& \leq C(p)\epsilon^{\sigma p}\int_{\mathcal{D}}\left(\sum_{i\in\mathbb{N}^d}\int_0^T {\left(e^{-\lambda_i^2(T-s)\epsilon}\right)^{2}}ds\right)^{\frac{p}{2}}dx\\
&\leq C(p)\epsilon^{\sigma p}\int_{\mathcal{D}}\left(\sum_{i\in\mathbb{N}^d}\int_0^T e^{-2\lambda_i^2(T-s)\epsilon}ds\right)^{\frac{p}{2}}\\
& \leq C(p)\epsilon^{(\sigma-\frac{1}{2})p}\int_{\mathcal{D}}\left(\sum_{i\in\mathbb{N}^d}\frac{1}{\lambda_i^2}\right)^{\frac{p}{2}}dx\leq C(p)\epsilon^{(\sigma-\frac{1}{2})p},
\end{align*}
where in the last step we used the fact that  $\sum_{i\in\mathbb{N}^d}\lambda_i^{-2}<\infty$, since $d=2,3$, see \eqref{convergencecondition}.
\end{proof}
 
 \begin{lemma}
 \label{bruitlemma1}
 For any $\vartheta>0$,  $p\geq 2$, there is a constant $C\geq 0$ such that
 \begin{align*}
 \mathbb{E}\left[\sup_{t\in[0, T]}\Vert (-\Delta)^{1-\frac{d}{4}-\frac{\vartheta}{2}}Z^{\epsilon}(t)\Vert^p\right]\leq C\epsilon^{\left(\sigma-\frac{1}{2}\right)p}.
 \end{align*}
 \end{lemma}
 \begin{proof}
From \eqref{Convolution} and \eqref{Fractionaire1}, it follows that
\begin{align}
\label{Convolution1} 
(-\Delta)^{1-\frac{d}{4}-\frac{\vartheta}{2}}Z^{\epsilon}(t)=\epsilon^{\sigma}\sum_{i\in\mathbb{N}^d}\int_0^t\lambda_i^{1-\frac{d}{4}-\frac{\vartheta}{2}}e^{-\lambda_i^2(t-s)\epsilon}e_id\beta_i(s),\quad t\in[0, T]. 
\end{align}
Taking the $\mathbb{L}^2$-norm in \eqref{Convolution1}, raising to power $p$, using the embedding $\mathbb{L}^p\hookrightarrow \mathbb{L}^2$ ($p\geq 2$), taking the supremum, the expectation  and using the BDG inequality \cite[Proposition 4.36]{DaPratoZabczyk}, it follows that
\begin{align*}
&\mathbb{E}\left[\sup_{t\in[0, T]}\Vert (-\Delta)^{1-\frac{d}{4}-\frac{\vartheta}{2}} Z^{\epsilon}(t)\Vert^p\right]
\leq C\mathbb{E}\left[\sup_{t\in[0, T]}\Vert (-\Delta)^{1-\frac{d}{4}-\frac{\vartheta}{2}} Z^{\epsilon}(t)\Vert^p_{\mathbb{L}^p}\right]\nonumber\\
&\leq C\epsilon^{\sigma p}\int_{\mathcal{D}}\mathbb{E}\left[\sup_{t\in[0, T]}\left\vert\sum_{i\in\mathbb{N}^d}\int_0^t\lambda_i^{1-\frac{d}{4}-\frac{\vartheta}{2}}e^{-\lambda_i^2(t-s)\epsilon}d\beta_i(s)e_i(x)\right\vert^p\right]dx\nonumber\\
&\leq C\epsilon^{\sigma p}\int_{\mathcal{D}}\left(\sum_{i\in\mathbb{N}^d}\int_0^T  \lambda_i^{2-\frac{d}{2}-\vartheta} e^{-2\lambda_i^2(T-s)\epsilon}ds\right)^{\frac{p}{2}}dx\nonumber\\
&\leq C\epsilon^{(\sigma-\frac{1}{2})p}\int_{\mathcal{D}}\left(\sum_{i\in\mathbb{N}^d}\lambda_i^{-\frac{d}{2}-\vartheta}\right)^{\frac{p}{2}}dx\leq C\epsilon^{(\sigma-\frac{1}{2})p},
\end{align*}
where in the last step we used the fact that  $\sum_{i\in\mathbb{N}^d}\lambda_i^{-\frac{d}{2}-\vartheta}<\infty$,  see \eqref{convergencecondition}.
 \end{proof}

 \begin{theorem}
 \label{regularitytheorem}
 Let $u^{\epsilon}_0\in \mathbb{H}^{-1}$, then there exists a unique strong variational solution $u^{\epsilon}$ of \eqref{sch}, such that 
 \begin{align*}
 u^{\epsilon}\in L^2\left(\Omega; C([0, T]; \mathbb{H}^{-1})\right)\cap L^2\left(\Omega; L^2(0, T; \mathbb{H}^{2-\frac{d}{2}-\vartheta})\right)\cap L^4\left(\Omega; L^4(0, T; \mathbb{L}^4)\right),
 \end{align*}  
   Furthermore, for any $p\geq 2$ it holds that
  \begin{align}
  \label{Regularesti}
 \mathcal{E}_p(u^{\epsilon})&:= \mathbb{E}\left[\Vert u^{\epsilon}\Vert^p_{L^{\infty}(0, T; \mathbb{H}^{-1})}+\epsilon^{\frac{p}{2}}\Vert (-\Delta)^{1-\frac{d}{4}-\frac{\vartheta}{2}}u^{\epsilon}\Vert^p_{L^2(0, T; \mathbb{L}^2)}+\frac{1}{\epsilon^{\frac{p}{2}}}\Vert u^{\epsilon}\Vert^{2p}_{L^4(0, T; \mathbb{L}^4)}\right]\nonumber\\
 &\leq C\left(\epsilon^{-\frac{p}{2}}+\epsilon^{\left(\sigma-\frac{1}{2}\right)p}+\epsilon^{\left(2\sigma-\frac{3}{2}\right)p}\right),
  \end{align}
  where $\vartheta>0$ is any  arbitrary small number.
 \end{theorem}

 \begin{proof}
 The proof of the existence and the uniqueness, as well as the proof of the fact that $u^{\epsilon}$ belongs to $L^2\left(\Omega; C([0, T]; \mathbb{H}^{-1})\right)$ can be found in \cite[Theorem 2.1]{Debussche1}.  To prove that $u^{\epsilon}$ belongs to $L^4\left(\Omega; L^4(0, T; \mathbb{L}^4)\right)$ and to $L^2\left(\Omega; L^2(0, T; \mathbb{H}^{2-\frac{d}{2}-\vartheta})\right)$, we set $\widehat{u}^{\epsilon}(t):=u^{\epsilon}(t)-Z^{\epsilon}(t)$. Then $\widehat{u}^{\epsilon}(t)$ satisfies the following random PDE
\begin{align*}
\frac{d}{dt}\widehat{u}^{\epsilon}(t)&=-\epsilon\Delta^2\widehat{u}^{\epsilon}(t)+\frac{1}{\epsilon}\Delta f(\widehat{u}^{\epsilon}(t)+Z^{\epsilon}(t))\quad t\in (0, T],\\
\widehat{u}^{\epsilon}(0)&=u^{\epsilon}_0.
\end{align*}
Testing the above equation with $(-\Delta)^{-1}\widehat{u}^{\epsilon}(t)$ yields
\begin{align*}
\frac{1}{2}\frac{d}{dt}\Vert \widehat{u}^{\epsilon}(t)\Vert^2_{\mathbb{H}^{-1}}+\epsilon\Vert\nabla \widehat{u}^{\epsilon}(t)\Vert^2+\frac{1}{\epsilon}\left(f(\widehat{u}^{\epsilon}(t)+Z^{\epsilon}(t)), \widehat{u}^{\epsilon}(t)\right)=0.
\end{align*}
Using the  fact that $(f(v), v)\geq \frac{1}{2}\Vert v\Vert^4_{\mathbb{L}^4}-C$, $v\in \mathbb{L}^4$, it follows that
\begin{align}
\label{DaPratoZabczyk1}
&\frac{1}{2}\frac{d}{dt}\Vert \widehat{u}^{\epsilon}(t)\Vert^2_{\mathbb{H}^{-1}}+\epsilon\Vert\nabla \widehat{u}^{\epsilon}(t)\Vert^2+\frac{1}{2\epsilon}\Vert \widehat{u}^{\epsilon}(t)+Z^{\epsilon}(t)\Vert^4_{\mathbb{L}^4}\nonumber\\
&\leq \frac{C}{\epsilon}+\frac{1}{\epsilon}\left\vert \left(f(\widehat{u}^{\epsilon}(t)+Z^{\epsilon}(t)), Z^{\epsilon}(t)\right)\right\vert.
\end{align}
Noting that $\vert f(x)\vert\leq 2\vert x\vert^3+C_1$, using H\"{o}lder and Young's inequalities and the embbeding $\mathbb{L}^4\hookrightarrow \mathbb{L}^1$, we deduce that
\begin{align}
\label{pra1}
&\left\vert \left(f(\widehat{u}^{\epsilon}(t)+Z^{\epsilon}(t)), Z^{\epsilon}(t)\right)\right\vert\nonumber\\
&\leq2\int_{\mathcal{D}}\vert\widehat{u}^{\epsilon}(t)+Z^{\epsilon}(t)\vert^3\vert Z^{\epsilon}(t)\vert dx+C_1\int_{\mathcal{D}}\vert Z^{\epsilon}(t)\vert dx\nonumber\\
&\leq 2\left(\int_{\mathcal{D}}\vert \widehat{u}^{\epsilon}(t)+Z^{\epsilon}(t)\vert^4 dx\right)^{\frac{3}{4}}\left(\int_{\mathcal{D}}\vert Z^{\epsilon}(t)\vert^4 dx\right)^{\frac{1}{4}}+C_1\int_{\mathcal{D}}\vert Z^{\epsilon}(t)\vert dx\\
&\leq \frac{1}{4}\int_{\mathcal{D}}\vert \widehat{u}^{\epsilon}(t)+Z^{\epsilon}(t)\vert^4dx+C\int_{\mathcal{D}}\vert Z^{\epsilon}(t)\vert^4dx+C_1\int_{\mathcal{D}}\vert Z^{\epsilon}(t)\vert dx\nonumber\\
&\leq \frac{1}{4}\Vert \widehat{u}^{\epsilon}(t)+Z^{\epsilon}(t)\Vert^4_{\mathbb{L}^4}+C\Vert Z^{\epsilon}(t)\Vert^4_{\mathbb{L}^4}+C.\nonumber
\end{align}
Substituting \eqref{pra1} into \eqref{DaPratoZabczyk1} and absorbing  $\frac{1}{4\epsilon}\Vert \widehat{u}^{\epsilon}(t)+Z^{\epsilon}(t)\Vert^4_{\mathbb{L}^4}$ into the left hand side, yields
\begin{align}
\label{pra1a}
\frac{1}{2}\frac{d}{dt}\Vert \widehat{u}^{\epsilon}(t)\Vert^2_{\mathbb{H}^{-1}}+\epsilon\Vert\nabla \widehat{u}^{\epsilon}(t)\Vert^2+\frac{1}{4\epsilon}\Vert \widehat{u}^{\epsilon}(t)+Z^{\epsilon}(t)\Vert^4_{\mathbb{L}^4}\leq \frac{C}{\epsilon}+\frac{C}{\epsilon}\Vert Z^{\epsilon}(t)\Vert^4_{\mathbb{L}^4}.
\end{align}
Integrating \eqref{pra1a} over $[0, t]$ and  taking the supremum over $[0, T]$ yields
\begin{align}
\label{depart1}
&\sup_{t\in[0, T]}\Vert\widehat{u}^{\epsilon}(t)\Vert^2_{\mathbb{H}^{-1}}+\epsilon\int_0^T\Vert\nabla\widehat{u}^{\epsilon}(s)\Vert^2ds+\frac{1}{4\epsilon}\int_0^T\Vert\widehat{u}^{\epsilon}(s)+Z^{\epsilon}(s)\Vert^4_{\mathbb{L}^4}ds\nonumber\\
&\leq  \Vert \widehat{u}^{\epsilon}(0)\Vert^2_{\mathbb{H}^{-1}}  +\frac{C}{\epsilon}+\frac{C}{\epsilon}\int_0^T\Vert Z^{\epsilon}(s)\Vert^4_{\mathbb{L}^4}ds.
\end{align}  
Taking the expectation on both sides of \eqref{depart1} and  using \lemref{convollemma} yields
\begin{align}
\label{depart1a}
&\mathbb{E}\left[\sup_{t\in[0, T]}\Vert \widehat{u}^{\epsilon}(t)\Vert^2_{\mathbb{H}^{-1}}\right]+\epsilon\int_0^T\mathbb{E}\left[\Vert\nabla \widehat{u}^{\epsilon}(s)\Vert^2\right]ds+\frac{1}{4\epsilon}\int_0^T\mathbb{E}\left[\Vert \widehat{u}^{\epsilon}(s)+Z^{\epsilon}(s)\Vert^4_{\mathbb{L}^4}\right]ds\nonumber\\
&\leq \Vert u^{\epsilon}_0\Vert^2_{\mathbb{H}^{-1}}+\frac{CT}{\epsilon}+\frac{C}{\epsilon}\int_0^T\mathbb{E}\left[\Vert Z^{\epsilon}(s)\Vert^4_{\mathbb{L}^4}\right]ds\leq C(1+\epsilon^{-1}+\epsilon^{4\sigma-3}). 
\end{align}
The proof of the fact that $u^{\epsilon}\in L^4\left(\Omega; L^4(0, T; \mathbb{L}^4)\right)$ then follows from \eqref{depart1a} by using triangle inequality and \lemref{convollemma}. 

 Using the fact  that $(-\Delta)^{-\alpha}$ is bounded in $\mathbb{L}^2$ for $\alpha>0$, the equivalence of norms $\Vert.\Vert_s$ and the usual norm of $\mathbb{H}^s$ for $s>0$ (see \secref{sec_notation}) and Poincar\'{e}'s inequality, it follows that
\begin{align}
\label{depart1b}
\Vert (-\Delta)^{1-\frac{d}{4}-\frac{\vartheta}{2}}\widehat{u}(s)\Vert&\leq \Vert (-\Delta)^{\frac{1}{2}-\frac{d}{4}-\frac{\vartheta}{2}}\Vert_{\mathcal{L}(\mathbb{L}^2)}\Vert (-\Delta)^{\frac{1}{2}}\widehat{u}(s)\Vert\nonumber\\
&\leq C\Vert (-\Delta)^{\frac{1}{2}}\widehat{u}(s)\Vert\leq C\Vert\nabla\widehat{u}(s)\Vert,
\end{align}
where we used   that $\frac{1}{2}-\frac{d}{4}-\frac{\vartheta}{2}<0$ (since $d=2,3$ and $\vartheta>0$).

The proof of the fact that $u^{\epsilon}\in L^2\left(\Omega; L^2(0, T; \mathbb{H}^{2-\frac{d}{2}-\vartheta})\right)$ follows from \eqref{depart1a} and \eqref{depart1b} by using triangle inequality and \lemref{bruitlemma1}.

 To prove \eqref{Regularesti}, we start from \eqref{depart1}. Omitting the terms involving the norms $\Vert\nabla .\Vert$ and $\Vert.\Vert_{\mathbb{L}^4}$ on the left hand side, raising the resulting inequality to power $\frac{p}{2}$, taking 
the expectation on both sides, using H\"{o}lder's inequality, the embedding $\mathbb{L}^r\hookrightarrow \mathbb{L}^s$, $s\leq r$, and \lemref{convollemma} yields
\begin{align*}
\mathbb{E}\left[\Vert \widehat{u}^{\epsilon}\Vert^p_{\mathbb{L}^{\infty}(0, T; \mathbb{H}^{-1})}\right]&\leq \Vert u^{\epsilon}_0\Vert^2_{\mathbb{H}^{-1}}+ C\epsilon^{-\frac{p}{2}}+\frac{C}{\epsilon^{\frac{p}{2}}} \mathbb{E}\left(\int_0^T\Vert Z^{\epsilon}(s)\Vert^4_{\mathbb{L}^4}ds\right)^{\frac{p}{2}}\nonumber\\
&\leq  \Vert u^{\epsilon}_0\Vert^2_{\mathbb{H}^{-1}}+ C\epsilon^{-\frac{p}{2}}+\frac{C}{\epsilon^{\frac{p}{2}}}\int_0^T\mathbb{E}\Vert Z^{\epsilon}(s)\Vert^{2p}_{\mathbb{L}^4}ds\nonumber\\
&\leq  \Vert u^{\epsilon}_0\Vert^2_{\mathbb{H}^{-1}}+ C\epsilon^{-\frac{p}{2}}+\frac{C}{\epsilon^{\frac{p}{2}}}\int_0^T\mathbb{E}\Vert Z^{\epsilon}(s)\Vert^{2p}_{\mathbb{L}^{2p}}ds\nonumber\\
&\leq  \Vert u^{\epsilon}_0\Vert^2_{\mathbb{H}^{-1}}+ C\epsilon^{-\frac{p}{2}}+C\epsilon^{\left(2\sigma-\frac{3}{2}\right)p}. 
\end{align*}
Repeating the argument above (i.e., dropping appropriate terms on the left hand-side in \eqref{depart1}, raising the resulting inequality to power $\frac{p}{2}$, using H\"{o}lder's inequality, \eqref{depart1b} and \lemref{convollemma}), we arrive at
\begin{align*}
\frac{1}{\epsilon^{\frac{p}{2}}}\mathbb{E}\left[\Vert\widehat{u}^{\epsilon}+Z^{\epsilon}\Vert^{2p}_{L^4(0, T; \mathbb{L}^4)}+ \epsilon^{\frac{p}{2}}\Vert (-\Delta)^{1-\frac{d}{4}-\frac{\vartheta}{2}}\widehat{u}^{\epsilon}\Vert^p_{L^2(0, T, \mathbb{L}^2)}\right]\leq  \Vert u^{\epsilon}_0\Vert^2_{\mathbb{H}^{-1}}+ C\epsilon^{-\frac{p}{2}}+C\epsilon^{\left(2\sigma-\frac{3}{2}\right)p}. 
\end{align*}
Summing the  two preceding estimates, using Lemma \ref{convollemma} and \ref{bruitlemma1} completes the proof of \eqref{Regularesti}. 
%The proof of \thmref{regularitytheorem} is thus completed. 
\end{proof}

%%%%%%%%%%%%%%%%%%%%%%%%%%%%%%%%%%%%%%%%%%%%%%%%%%%%%%%%%%%%%%%%%%%%%%%%%%%%%%%%%%%%%%%%%%%%%%%%%%%%%%%%%%%%%%%%
%%%%%%%%%%%%%%%%%%%%%%%%%%%%%%%%%%%%%%%%%%%%%%%%%%%%%%%%%%%%%%%%%%%%%%%%%%%%%%%%%%%%%%%%%%%%%%%%%%%%%%%%%%%%%%%%
%%%%%%%%%%%%%%%%%%%%%%%%%%%%%%%%%%%%%%%%%%%%%%%%%%%%%%%%%%%%%%%%%%%%%%%%%%%%%%%%%%%%%%%%%%%%%%%%%%%%%%%%%%%%%%%%
%%%%%%%%%%%%%%%%%%%%%%%%%%%%%%%%%%%%%%%%%%%%%%%%%%%%%%%%%%%%%%%%%%%%%%%%%%%%%%%%%%%%%%%%%%%%%%%%%%%%%%%%%%%%%%%%
%%%%%%%%%%%%%%%%%%%%%%%%%%%%%%%%%%%%%%%%%%%%%%%%%%%%%%%%%%%%%%%%%%%%%%%%%%%%%%%%%%%%%%%%%%%%%%%%%%%%%%%%%%%%%%%%
%%%%%%%%%%%%%%%%%%%%%%%%%%%%%%%%%%%%%%%%%%%%%%%%%%%%%%%%%%%%%%%%%%%%%%%%%%%%%%%%%%%%%%%%%%%%%%%%%%%%%%%%%%%%%%%%

\section{Sharp interface limit of the stochastic problem}
\label{sec_stoch}

Recall that the solution $u^{\epsilon}_{\mathsf{A}}$ of \eqref{model5} is constructed in \cite{abc94}.
We set $R^{\epsilon}:=u^{\epsilon}-u^{\epsilon}_{\mathsf{A}}$. From \eqref{sch}
and \eqref{model5}, it follows that  $R^{\epsilon}$ satisfies the stochastic PDE (SPDE)
\begin{align}
\begin{array}{ll}
dR^{\epsilon}=-\epsilon\Delta^2R^{\epsilon}dt+\frac{1}{\epsilon}\Delta\left(f(u^{\epsilon}_{\mathsf{A}}+R^{\epsilon})-f(u^{\epsilon}_{\mathsf{A}})\right)dt -\Delta r^{\epsilon}_{\mathsf{A}}dt+\epsilon^{\sigma}dW &\text{in}\; \mathcal{D}_T,
\\
{\partial_{\vec{n}} R^{\epsilon}}={\partial_{\vec{n}} \Delta R^{\epsilon}}=0 &\text{on}\; \partial\mathcal{D},\\
R^{\epsilon}(0)=0 &\text{in}\; \mathcal{D}.
\end{array}
\end{align}
We set $Y^{\epsilon}:=R^{\epsilon}-Z^{\epsilon}$. Note that \eqref{Convolution} implies that $dZ^{\epsilon}=-\epsilon\Delta^2Z^{\epsilon}+\epsilon^{\sigma}dW$.  
Hence, we deduce that $Y^{\epsilon}$ satisfies $\mathbb{P}$-a.s. the following random PDE
\begin{align}
\label{Yepsilon1}
\begin{array}{ll}
\frac{d}{dt} Y^{\epsilon}=-\epsilon\Delta^2Y^{\epsilon}+\frac{1}{\epsilon}\Delta\left(f(u^{\epsilon}_{\mathsf{A}}+Y^{\epsilon}+Z^{\epsilon})-f(u^{\epsilon}_{\mathsf{A}})\right)  -\Delta r^{\epsilon}_{\mathsf{A}} &\text{in}\; \mathcal{D}_T,\\
{\partial_{\vec{n}} Y^{\epsilon}}={\partial_{\vec{n}} \Delta Y^{\epsilon}}=0& \text{on}\; \partial\mathcal{D},\\
Y^{\epsilon}(0)=0 &\text{in}\; \mathcal{D}.
\end{array}
\end{align}

In the next lemma we derive an estimate for the solution of the RPDE \eqref{Yepsilon1}.
\begin{lemma}
\label{Formulalemma}
The following estimate holds $\mathbb{P}$-a.s. for the solution of (\ref{Yepsilon1})
\begin{align*}
&\Vert Y^{\epsilon}(t)\Vert^2_{\mathbb{H}^{-1}}+\epsilon^{ 4}\int_0^t\Vert \nabla Y^{\epsilon}(s)\Vert^2ds+\frac{13}{8\epsilon}\int_0^t\Vert Y^{\epsilon}(s)\Vert^4_{\mathbb{L}^4}ds\nonumber\\
&\leq\frac{C}{\epsilon}\int_0^t\Vert Y^{\epsilon}(s)\Vert^3_{\mathbb{L}^3}ds+\frac{C}{\epsilon}\int_0^t\Vert Z^{\epsilon}(s)\Vert^{\frac{4}{3}}_{\mathbb{L}^{\frac{4}{3}}}ds+\frac{C}{\epsilon}\int_0^t\Vert Z^{\epsilon}(s)\Vert^2ds\nonumber\\
&\quad+\frac{C}{\epsilon}\int_0^t\Vert Z^{\epsilon}(s)\Vert^{\frac{8}{3}}_{\mathbb{L}^{\frac{8}{3}}}ds+\frac{C}{\epsilon}\int_0^t\Vert Z^{\epsilon}(s)\Vert^4_{\mathbb{L}^4}ds+ C\epsilon^{\frac{1}{2}}\int_0^t\Vert r^{\epsilon}_{\mathsf{A}}(s)\Vert^{\frac{3}{2}}_{C(\mathcal{D})} ds\quad  \text{ for }  t\in[0, T].
\end{align*}
\end{lemma}

\begin{proof}
We fix $\omega \in \Omega$ and consider $Y^{\epsilon} \equiv Y^{\epsilon}(\omega)$.
Testing \eqref{Yepsilon1} with $(-\Delta)^{-1}Y^{\epsilon}$ yields
\begin{align}
\label{formula1}
\frac{1}{2}\frac{d}{dt}\Vert Y^{\epsilon}\Vert^2_{\mathbb{H}^{-1}}+\epsilon\Vert \nabla Y^{\epsilon}(t)\Vert^2+\frac{1}{\epsilon}\left( f(u^{\epsilon}_{\mathsf{A}}+Y^{\epsilon}+Z^{\epsilon})-f(u^{\epsilon}_{\mathsf{A}}), Y^{\epsilon}\right)  -\left( r^{\epsilon}_{\mathsf{A}}, Y^{\epsilon}\right)=0.
\end{align}
Recall that $f(s)=s^3-s$.  A straightforward computation yields
\begin{align}
\label{identity1}
f(a)-f(b)=(a-b)f'(a)+(a-b)^3-3(a-b)^2a,\quad a,b\in\mathbb{R}.
\end{align}
Using \eqref{identity1} we obtain
\begin{align*}
&\left( f(u^{\epsilon}_{\mathsf{A}}+Y^{\epsilon}+Z^{\epsilon})-f(u^{\epsilon}_{\mathsf{A}}+Z^{\epsilon}), Y^{\epsilon}\right)\nonumber\\
&=-\left( f(u^{\epsilon}_{\mathsf{A}}+Z^{\epsilon})-f(u^{\epsilon}_{\mathsf{A}}+Y^{\epsilon}+Z^{\epsilon}), Y^{\epsilon}\right)\nonumber\\
&=\left( f'(u^{\epsilon}_{\mathsf{A}}+Z^{\epsilon})Y^{\epsilon}, Y^{\epsilon}\right)+\Vert Y^{\epsilon}\Vert^4_{\mathbb{L}^4}+3\left( (Y^{\epsilon})^3, u^{\epsilon}_{\mathsf{A}}+Z^{\epsilon}\right)\nonumber\\
&=\left( f'(u^{\epsilon}_{\mathsf{A}})Y^{\epsilon}, Y^{\epsilon}\right)+\left( \left(f'(u^{\epsilon}_{\mathsf{A}}+Z^{\epsilon})-f'(u^{\epsilon}_{\mathsf{A}})\right)Y^{\epsilon}, Y^{\epsilon}\right)+\Vert Y^{\epsilon}\Vert^4_{\mathbb{L}^4}+3\left( (Y^{\epsilon})^3, u^{\epsilon}_{\mathsf{A}}+Z^{\epsilon}\right).
\end{align*}
Using the preceding identity we rewrite
\begin{align*}
&\left( f(u^{\epsilon}_{\mathsf{A}}+Y^{\epsilon}+Z^{\epsilon})-f(u^{\epsilon}_{\mathsf{A}}), Y^{\epsilon}\right)\nonumber\\
&=\left(f(u^{\epsilon}_{\mathsf{A}}+Y^{\epsilon}+Z^{\epsilon})-f(u^{\epsilon}_{\mathsf{A}}+Z^{\epsilon}), Y^{\epsilon}\right)+\left( f(u^{\epsilon}_{\mathsf{A}}+Z^{\epsilon})-f(u^{\epsilon}_{\mathsf{A}}), Y^{\epsilon}\right)\nonumber\\
&=\left( f'(u^{\epsilon}_{\mathsf{A}})Y^{\epsilon}, Y^{\epsilon}\right)+\left(\left(f'(u^{\epsilon}_{\mathsf{A}}+Z^{\epsilon})-f'(u^{\epsilon}_{\mathsf{A}})\right)Y^{\epsilon}, Y^{\epsilon}\right)+\Vert Y^{\epsilon}\Vert^4_{\mathbb{L}^4}\nonumber\\
&\quad+3\left( (Y^{\epsilon})^3, u^{\epsilon}_{\mathsf{A}}+Z^{\epsilon}\right)+\left( f(u^{\epsilon}_{\mathsf{A}}+Z^{\epsilon})-f(u^{\epsilon}_{\mathsf{A}}), Y^{\epsilon}\right).
\end{align*}
Substituting the identity above into \eqref{formula1} leads to
\begin{align}
\label{formula2}
&\frac{1}{2}\frac{d}{dt}\Vert Y^{\epsilon}(t)\Vert^2_{\mathbb{H}^{-1}}+\epsilon\Vert \nabla Y^{\epsilon}(t)\Vert^2+\frac{1}{\epsilon}\left( f'(u^{\epsilon}_{\mathsf{A}})Y^{\epsilon}, Y^{\epsilon}\right)+\frac{1}{\epsilon}\Vert Y^{\epsilon}\Vert^4_{\mathbb{L}^4}\nonumber\\
&\leq \frac{1}{\epsilon}\left\vert \left( \left(f'(u^{\epsilon}_{\mathsf{A}}+Z^{\epsilon})-f'(u^{\epsilon}_{\mathsf{A}})\right)Y^{\epsilon}, Y^{\epsilon}\right)\right\vert+\frac{3}{\epsilon}\left\vert \left( (Y^{\epsilon})^3, u^{\epsilon}_{\mathsf{A}}+Z^{\epsilon}\right)\right\vert\\
&\quad+\frac{1}{\epsilon}\left\vert \left( f(u^{\epsilon}_{\mathsf{A}}+Z^{\epsilon})-f(u^{\epsilon}_{\mathsf{A}}), Y^{\epsilon}\right)\right\vert+\vert (r^{\epsilon}_{\mathsf{A}}, Y^{\epsilon})\vert\nonumber\\
&=:I+II+III+IV\nonumber. 
\end{align}
Noting  $f'(a)-f'(b)= 3(a-b)(a+b)$, using the uniform boundedness of $u^{\epsilon}_{\mathsf{A}}$ (cf. \eqref{Alakiresult1}), H\"{o}lder's and Young's inequalities yields
\begin{align*}
I&=\frac{1}{\epsilon}\left\vert \left( \left(f'(u^{\epsilon}_{\mathsf{A}}+Z^{\epsilon})-f'(u^{\epsilon}_{\mathsf{A}})\right)Y^{\epsilon}, Y^{\epsilon}\right)\right\vert\leq \frac{C}{\epsilon}\int_{\mathcal{D}}\vert Z^{\epsilon}\vert(Y^{\epsilon})^2dx+\frac{C}{\epsilon}\int_{\mathcal{D}}(Z^{\epsilon})^2(Y^{\epsilon})^2dx\nonumber\\
&\leq \frac{C}{\epsilon}\Vert Y^{\epsilon}\Vert^2_{\mathbb{L}^4}\Vert Z^{\epsilon}\Vert+\frac{C}{\epsilon}\Vert Y^{\epsilon}\Vert^2_{\mathbb{L}^4}\Vert Z^{\epsilon}\Vert^2_{\mathbb{L}^4}\leq \frac{1}{16\epsilon}\Vert Y^{\epsilon}\Vert^4_{\mathbb{L}^4}+\frac{C}{\epsilon}\Vert Z^{\epsilon}\Vert^2+\frac{C}{\epsilon}\Vert Z^{\epsilon}\Vert^4_{\mathbb{L}^4}.
\end{align*} 
Using the uniform boundedness of $u^{\epsilon}_{\mathsf{A}}$  \eqref{Alakiresult1}, H\"{o}lder's and Young's inequalities leads to
\begin{align*}
II&=\frac{3}{\epsilon}\left\vert \left( (Y^{\epsilon})^3, u^{\epsilon}_{\mathsf{A}}+Z^{\epsilon}\right)\right\vert\leq \frac{C}{\epsilon}\Vert Y^{\epsilon}\Vert^3_{\mathbb{L}^3}+\frac{3}{\epsilon}\left\vert \left( (Y^{\epsilon})^3, Z^{\epsilon}\right)\right\vert\nonumber\\
&\leq  \frac{C}{\epsilon}\Vert Y^{\epsilon}\Vert^3_{\mathbb{L}^3}+\frac{3}{\epsilon}\Vert Y^{\epsilon}\Vert^3_{\mathbb{L}^4}\Vert Z^{\epsilon}\Vert_{\mathbb{L}^4}\leq \frac{C}{\epsilon}\Vert Y^{\epsilon}\Vert^3_{\mathbb{L}^3}+\frac{1}{16\epsilon}\Vert Y^{\epsilon}\Vert^4_{\mathbb{L}^4}+\frac{C}{\epsilon}\Vert Z^{\epsilon}\Vert^4_{\mathbb{L}^4}.
\end{align*}
Using \eqref{identity1},  \eqref{Alakiresult1}, H\"{o}lder's and Young's inequalities yields
\begin{align*}
III&=\frac{1}{\epsilon}\left\vert \left( f(u^{\epsilon}_{\mathsf{A}}+Z^{\epsilon})-f(u^{\epsilon}_{\mathsf{A}}), Y^{\epsilon}\right)\right\vert \nonumber\\
&\leq \frac{1}{\epsilon}\int_{\mathcal{D}}\vert Z^{\epsilon}\vert\vert f'(u^{\epsilon}_{\mathsf{A}})\vert\vert Y^{\epsilon}\vert +\frac{1}{\epsilon}\int_{\mathcal{D}}\vert Z^{\epsilon}\vert^3 \vert Y^{\epsilon}\vert +\frac{3}{\epsilon}\int_{\mathcal{D}}\vert Z^{\epsilon}\vert^2\vert u^{\epsilon}_{\mathsf{A}}\vert \vert Y^{\epsilon}\vert\nonumber\\
&\leq \frac{C}{\epsilon}\Vert Y^{\epsilon}\Vert_{\mathbb{L}^4}\Vert Z^{\epsilon}\Vert_{\mathbb{L}^{\frac{4}{3}}}+\Vert \frac{1}{\epsilon}\Vert Y^{\epsilon}\Vert_{\mathbb{L}^4}\Vert Z^{\epsilon}\Vert^{3}_{\mathbb{L}^4}+\frac{C}{\epsilon}\Vert Y^{\epsilon}\Vert_{\mathbb{L}^4}\Vert Z^{\epsilon}\Vert^2_{\mathbb{L}^{\frac{8}{3}}}\nonumber\\
&\leq \frac{1}{16\epsilon}\Vert Y^{\epsilon}\Vert^4_{\mathbb{L}^4}+\frac{C}{\epsilon}\Vert Z^{\epsilon}\Vert^{\frac{4}{3}}_{\mathbb{L}^{\frac{4}{3}}}+\frac{C}{\epsilon}\Vert Z^{\epsilon}\Vert^{\frac{8}{3}}_{\mathbb{L}^{\frac{8}{3}}}+\frac{C}{\epsilon}\Vert Z^{\epsilon}\Vert^4_{\mathbb{L}^4}. 
\end{align*}
Using  the embedding $\mathbb{L}^3 \hookrightarrow\mathbb{L}^1$ and Young's inequality, it follows that
\begin{align*}
IV=\vert\left( r^{\epsilon}_{\mathsf{A}}, Y^{\epsilon}(t)\right)\vert&\leq \Vert r^{\epsilon}_{\mathsf{A}}\Vert_{C(\mathcal{D})}\Vert Y^{\epsilon}(t)\Vert_{\mathbb{L}^1}\leq C\Vert r^{\epsilon}_{\mathsf{A}}\Vert_{C(\mathcal{D})}\Vert Y^{\epsilon}(t)\Vert_{\mathbb{L}^3}\nonumber\\
&\leq C\epsilon^{\frac{1}{2}}\Vert r^{\epsilon}_{\mathsf{A}}\Vert^{\frac{3}{2}}_{C(\mathcal{D})}+\frac{C}{\epsilon}\Vert Y^{\epsilon}(t)\Vert^3_{\mathbb{L}^3}.
\end{align*}
Substituting the above estimates of $I$, $II$, $III$ and $IV$ into \eqref{formula2} leads to
\begin{align}
\label{apriori1} 
&\frac{1}{2}\frac{d}{dt}\Vert Y^{\epsilon}\Vert^2_{\mathbb{H}^{-1}}+ \epsilon\Vert \nabla Y^{\epsilon}(t)\Vert^2+\frac{1}{\epsilon}\left( f'(u^{\epsilon}_{\mathsf{A}})Y^{\epsilon}, Y^{\epsilon}\right)+\frac{13}{16\epsilon}\Vert Y^{\epsilon}\Vert^4_{\mathbb{L}^4}\nonumber\\
&\leq \frac{C}{\epsilon}\Vert Y^{\epsilon}\Vert^3_{\mathbb{L}^3}+\frac{C}{\epsilon}\Vert Z^{\epsilon}\Vert^{\frac{4}{3}}_{\mathbb{L}^{\frac{4}{3}}}+\frac{C}{\epsilon}\Vert Z^{\epsilon}\Vert^2+\frac{C}{\epsilon}\Vert Z^{\epsilon}\Vert^{\frac{8}{3}}_{\mathbb{L}^{\frac{8}{3}}}+\frac{C}{\epsilon}\Vert Z^{\epsilon}\Vert^4_{\mathbb{L}^4}+ C\epsilon^{\frac{1}{2}}\Vert r^{\epsilon}_{\mathsf{A}}\Vert^{\frac{3}{2}}_{C(\mathcal{D})}.
\end{align}
Using \propref{Spectralestimate} we deduce from \eqref{apriori1} that
\begin{align}
\label{apriori1a}
\frac{1}{2}\frac{d}{dt}\Vert Y^{\epsilon}\Vert^2_{\mathbb{H}^{-1}}+\frac{13}{16\epsilon}\Vert Y^{\epsilon}\Vert^4_{\mathbb{L}^4}
\leq& \frac{C}{\epsilon}\left(\Vert Y^{\epsilon}\Vert^3_{\mathbb{L}^3}+\Vert Z^{\epsilon}\Vert^{\frac{4}{3}}_{\mathbb{L}^{\frac{4}{3}}}+\Vert Z^{\epsilon}\Vert^2+\Vert Z^{\epsilon}\Vert^{\frac{8}{3}}_{\mathbb{L}^{\frac{8}{3}}}+\Vert Z^{\epsilon}\Vert^4_{\mathbb{L}^4}\right)\nonumber\\
&+C\epsilon^{\frac{1}{2}}\Vert r^{\epsilon}_{\mathsf{A}}\Vert^{ \frac{3}{2}}_{C(\mathcal{D})}+C_0\Vert Y^{\epsilon}\Vert^2_{\mathbb{H}^{-1}}.
\end{align}
By noting that $f'(x)=3x^2-1$, it follows from \eqref{apriori1} that
\begin{align}
\label{apriori2} 
&\frac{1}{2}\frac{d}{dt}\Vert Y^{\epsilon}\Vert^2_{\mathbb{H}^{-1}}+ \epsilon\Vert \nabla Y^{\epsilon}(t)\Vert^2+\frac{13}{16\epsilon}\Vert Y^{\epsilon}\Vert^4_{\mathbb{L}^4}\nonumber\\
&\leq \frac{1}{\epsilon}\Vert Y^{\epsilon}\Vert^2+\frac{C}{\epsilon}\Vert Y^{\epsilon}\Vert^3_{\mathbb{L}^3}+\frac{C}{\epsilon}\Vert Z^{\epsilon}\Vert^{\frac{4}{3}}_{\mathbb{L}^{\frac{4}{3}}}+\frac{C}{\epsilon}\Vert Z^{\epsilon}\Vert^2+\frac{C}{\epsilon}\Vert Z^{\epsilon}\Vert^{\frac{8}{3}}_{\mathbb{L}^{\frac{8}{3}}}\nonumber\\
&\quad+\frac{C}{\epsilon}\Vert Z^{\epsilon}\Vert^4_{\mathbb{L}^4}+ C\epsilon^{\frac{1}{2}}\Vert r^{\epsilon}_{\mathsf{A}}\Vert^{\frac{3}{2}}_{C(\mathcal{D})}.
\end{align}
Multiplying \eqref{apriori2} by $\epsilon^3$ and \eqref{apriori1a} by $1-\epsilon^3$, summing up the resulting inequalities yields
\begin{align}
\label{apriori3} 
&\frac{1}{2}\frac{d}{dt}\Vert Y^{\epsilon}\Vert^2_{\mathbb{H}^{-1}}+ \epsilon^4\Vert \nabla Y^{\epsilon}(t)\Vert^2+\frac{13}{16\epsilon}\Vert Y^{\epsilon}\Vert^4_{\mathbb{L}^4}\nonumber\\
&\leq \epsilon^2\Vert Y^{\epsilon}\Vert^2+\frac{C}{\epsilon}\Vert Y^{\epsilon}\Vert^3_{\mathbb{L}^3}+\frac{C}{\epsilon}\Vert Z^{\epsilon}\Vert^{\frac{4}{3}}_{\mathbb{L}^{\frac{4}{3}}}+\frac{C}{\epsilon}\Vert Z^{\epsilon}\Vert^2_{\mathbb{L}^2}+\frac{C}{\epsilon}\Vert Z^{\epsilon}\Vert^{\frac{8}{3}}_{\mathbb{L}^{\frac{8}{3}}}\nonumber\\
&\quad+\frac{C}{\epsilon}\Vert Z^{\epsilon}\Vert^4_{\mathbb{L}^4}+C\epsilon^{\frac{1}{2}}\Vert r^{\epsilon}_{\mathsf{A}}\Vert^{\frac{3}{2}}_{C(\mathcal{D})}+ C_0\Vert Y^{\epsilon}\Vert^2_{\mathbb{H}^{-1}}.
\end{align}
Using the interpolation inequality $\Vert v\Vert^2\leq \Vert v\Vert_{\mathbb{H}^{-1}}\Vert \nabla v\Vert$, $v\in \mathbb{H}^1$, and Young's inequality, we estimate
\begin{align}
\label{apriori4}
\epsilon^2\Vert Y^{\epsilon}\Vert^2\leq \epsilon^2\Vert Y^{\epsilon}\Vert_{\mathbb{H}^{-1}}\Vert \nabla Y^{\epsilon}\Vert\leq \frac{1}{2}\Vert Y^{\epsilon}\Vert^2_{\mathbb{H}^{-1}}+\frac{\epsilon^4}{2}\Vert\nabla Y^{\epsilon}\Vert^2. 
\end{align}
Substituting \eqref{apriori4} into \eqref{apriori3} yields
\begin{align*}
%\label{apriori5} 
&\frac{1}{2}\frac{d}{dt}\Vert Y^{\epsilon}\Vert^2_{\mathbb{H}^{-1}}+ \frac{\epsilon^4}{2}\Vert \nabla Y^{\epsilon}(t)\Vert^2+\frac{13}{16\epsilon}\Vert Y^{\epsilon}\Vert^4_{\mathbb{L}^4}\nonumber\\
&\leq \frac{C}{\epsilon}\Vert Y^{\epsilon}\Vert^3_{\mathbb{L}^3}+\frac{C}{\epsilon}\Vert Z^{\epsilon}\Vert^{\frac{4}{3}}_{\mathbb{L}^{\frac{4}{3}}}+\frac{C}{\epsilon}\Vert Z^{\epsilon}\Vert^2_{\mathbb{L}^2}+\frac{C}{\epsilon}\Vert Z^{\epsilon}\Vert^{\frac{8}{3}}_{\mathbb{L}^{\frac{8}{3}}}+\frac{C}{\epsilon}\Vert Z^{\epsilon}\Vert^4_{\mathbb{L}^4}\\
&\quad+C\epsilon^{\frac{1}{2}}\Vert r^{\epsilon}_{\mathsf{A}}\Vert^{\frac{3}{2}}_{C(\mathcal{D})}+C\Vert Y^{\epsilon}\Vert^2_{\mathbb{H}^{-1}}.\nonumber
\end{align*}
Integrating the preceding inequality on $[0, t]$ and noting that $Y^{\epsilon}(0)=0$ leads to
\begin{align*}
&\Vert Y^{\epsilon}(t)\Vert^2_{\mathbb{H}^{-1}}+\epsilon^{4}\int_0^t\Vert \nabla Y^{\epsilon}(s)\Vert^2ds+\frac{13}{8\epsilon}\int_0^t\Vert Y^{\epsilon}(s)\Vert^4_{\mathbb{L}^4}ds\nonumber\\
&\leq\frac{C}{\epsilon}\left(\int_0^t\Vert Y^{\epsilon}(s)\Vert^3_{\mathbb{L}^3}ds+\int_0^t\Vert Z^{\epsilon}(s)\Vert^{\frac{4}{3}}_{\mathbb{L}^{\frac{4}{3}}}ds+\int_0^t\Vert Z^{\epsilon}(s)\Vert^2_{\mathbb{L}^2}ds+\int_0^t\Vert Z^{\epsilon}(s)\Vert^{\frac{8}{3}}_{\mathbb{L}^{\frac{8}{3}}}ds\right)\nonumber\\
&\quad+\frac{C}{\epsilon}\int_0^t\Vert Z^{\epsilon}(s)\Vert^4_{\mathbb{L}^4}ds+C_0\int_0^t\Vert Y^{\epsilon}(s)\Vert^2_{\mathbb{H}^{-1}}ds+C\epsilon^{\frac{1}{2}}\int_0^t\Vert r^{\epsilon}_{\mathsf{A}}(s)\Vert^{ \frac{3}{2}}_{C(\mathcal{D})} ds.
\end{align*}
The result follows after an application of Gronwall's lemma. 
\end{proof}

We introduce the space $\Omega_{\delta, \eta,\epsilon}\subset \Omega$, s.t. 
\begin{align}
\label{defOmega1}
\Omega_{\delta, \eta, \epsilon}=\left\{\omega \in \Omega:\, \Vert Z^{\epsilon}\Vert_{C(\mathcal{D}_T)}\leq C_1\epsilon^{\sigma^*-2\delta-2\eta}\right\},
\end{align}
 with $\sigma^*:=\sigma-\frac{1}{4}$. We also introduce the space $\widetilde{\Omega}_{\vartheta,\kappa,\epsilon}\subset\Omega$, s.t.,
 \begin{align*}
%\label{defOmega2}
\widetilde{\Omega}_{\vartheta,\kappa, \epsilon}=\left\{\omega\in\Omega:\, \sup_{t\in [0, T]}\Vert (-\Delta)^{1-\frac{d}{4}-\frac{\vartheta}{2}}Z^{\epsilon}(t)\Vert^2\leq C_1\epsilon^{2\gamma-\vartheta-\kappa-1}\right\}. 
\end{align*}

\begin{lemma}
\label{LuboLemma1}
For any $C_1, \delta>0$, $\eta\geq 0$, there exists a constant $C_{\delta, \eta}=C(\sigma, \eta, C_1)>0$, such that $\mathbb{P}[\Omega_{\delta,\eta, \epsilon}]>1-C_{\delta, \eta}\epsilon^{\delta+\eta}$. 
\end{lemma}

\begin{proof}
The proof is given in \cite[Lemma 4.6]{BYZ22} for $d=2$ but a close
inspection reveals that it is also valid for $d=3$. 
\end{proof}

\begin{lemma}
\label{Semigroup}
For any $C_1, \vartheta>0$, $\kappa\geq 0$, there exists a  constant $C_{\vartheta, \kappa}=C(\vartheta, \kappa, C_1)>0$  such that $\mathbb{P}[\widetilde{\Omega}_{\vartheta,\kappa, \epsilon}]\geq 1-C_{\vartheta, \kappa}\epsilon^{\vartheta+\kappa}$. 
\end{lemma}

\begin{proof}
Using Markov's inequality and \lemref{bruitlemma1} it holds that 
\begin{align*}
\mathbb{P}[\widetilde{\Omega}^c_{\vartheta, \kappa, \epsilon}]\leq \frac{\mathbb{E}\left[\sup_{t\in[0, T]}\Vert (-\Delta)^{1-\frac{d}{4}-\frac{\vartheta}{2}}Z^{\epsilon}(t)\Vert^2\right]}{C_1\epsilon^{2\sigma-\vartheta-\kappa-1}}\leq C_{\vartheta, \kappa}\epsilon^{\vartheta+\kappa}.
\end{align*}
The statement follows after noting that $\mathbb{P}[\widetilde{\Omega}_{\vartheta, \kappa, \epsilon}]=1-\mathbb{P}[\widetilde{\Omega}^c_{\vartheta, \kappa, \epsilon}]$. 
\end{proof}

We introduce the following stopping time 
\begin{align}
\label{stoppingtime}
T_{\epsilon}=T\wedge\inf\left\{t>0:\;\int_0^t\Vert Y^{\epsilon}(s)\Vert^3_{\mathbb{L}^3}ds>\epsilon^{\gamma}\right\},
\end{align}
for some $\gamma>0$, which will be specified later. 

In the next lemma we derive an estimate of $Y^{\epsilon}$ up to the stopping time $T_{\epsilon}$ on  $\Omega_{\delta,\eta, \epsilon}$.
\begin{lemma}
\label{lemmaerrorOmega2}
The following estimate holds for the solution of \eqref{Yepsilon1} for $\omega\in \Omega_{\delta,\eta,\epsilon}$ and $t\leq T_{\epsilon}$
\begin{align*}
&\sup_{s\in[0, t]}\Vert Y^{\epsilon}(s)\Vert^2_{\mathbb{H}^{-1}}+ \epsilon^4\int_0^t\Vert \nabla Y^{\epsilon}(s)\Vert^2ds+\frac{13}{8\epsilon}\int_0^t\Vert Y^{\epsilon}(s)\Vert^4_{\mathbb{L}^4}ds\nonumber\\
&\leq C\left(\epsilon^{\gamma-1}+\epsilon^{\frac{4}{3}(\sigma^*-2\delta-2\eta)-1}+\epsilon^{2(\sigma^*-2\delta-2\eta)-1}+\epsilon^{\frac{8}{3}(\sigma^*-2\delta-2\eta)-1}\right)\nonumber\\
&\quad+C\left(\epsilon^{4(\sigma^*-2\delta-2\eta)-1}+\epsilon^{ \frac{3K-5}{2}}\right),
\end{align*}
where $C$ is a positive constant independent of $\epsilon$ and $T_{\epsilon}$. 
\end{lemma}

\begin{proof}
 From   \lemref{Formulalemma},  using \eqref{Alakiresult1}, the embbeding $C(\mathcal{D})\hookrightarrow \mathbb{L}^q$, $q\geq 1$,   recalling the definitions of $T_{\epsilon}$ \eqref{stoppingtime} and $\Omega_{\delta, \eta, \epsilon}$ \eqref{defOmega1}, yields
 for any $t<T_{\epsilon}$ 
\begin{align*}
&\sup_{s\in[0, t]}\Vert Y^{\epsilon}(s)\Vert^2_{\mathbb{H}^{-1}}+ \epsilon^4\int_0^t\Vert \nabla Y^{\epsilon}(s)\Vert^2ds+\frac{13}{8\epsilon}\int_0^t\Vert Y^{\epsilon}(s)\Vert^4_{\mathbb{L}^4}ds\nonumber\\
&\leq\frac{C}{\epsilon}\int_0^t\Vert Y^{\epsilon}(s)\Vert^3_{\mathbb{L}^3}ds+\frac{C}{\epsilon}\int_0^t\Vert Z^{\epsilon}(s)\Vert^{\frac{4}{3}}_{\mathbb{L}^{\frac{4}{3}}}ds+\frac{C}{\epsilon}\int_0^t\Vert Z^{\epsilon}(s)\Vert^2_{\mathbb{L}^2}ds\nonumber\\
&\quad+\frac{C}{\epsilon}\int_0^t\Vert Z^{\epsilon}(s)\Vert^{\frac{8}{3}}_{\mathbb{L}^{\frac{8}{3}}}ds+\frac{C}{\epsilon}\int_0^t\Vert Z^{\epsilon}(s)\Vert^4_{\mathbb{L}^4}ds+ C\epsilon^{\frac{1}{2}}\int_0^t\Vert r^{\epsilon}_{\mathsf{A}}(s)\Vert^{\frac{3}{2}}_{C(\mathcal{D})}ds\nonumber\\
&\leq C\left(\epsilon^{\gamma-1}+\epsilon^{\frac{4}{3}(\sigma^*-2\delta-2\eta)-1}+\epsilon^{2(\sigma^*-2\delta-2\eta)-1}+\epsilon^{\frac{8}{3}(\sigma^*-2\delta-2\eta)-1}\right)\nonumber\\
&\quad+C\left(\epsilon^{4(\sigma^*-2\delta-2\eta)-1}+\epsilon^{\frac{3K-5}{2}}\right).
\end{align*}
\end{proof}

In order to show that $T_{\epsilon}\equiv T$  on $\Omega_{\delta, \eta, \epsilon}$, we make use of the following interpolation inequality.
\begin{lemma}
\label{Fundamentallemma}
For all $2<r<3$ and $\widetilde{C}>0$, there exists a positive constant $C_{\mathcal{D}}$, independent of $\epsilon$  such that for every $v\in \mathbb{H}^1\cap \mathbb{L}^2_0$ and $\alpha\in\mathbb{R}$, it holds
\begin{align*}
\widetilde{C}\Vert v\Vert^3_{\mathbb{L}^3}\leq \epsilon^{\alpha}\Vert v\Vert^4_{\mathbb{L}^4}+C_{\mathcal{D}}\frac{\widetilde{C}^{4-r}}{4-r}\epsilon^{-\alpha(3-r)}\Vert v\Vert^{\frac{4-r}{2}}_{\mathbb{H}^{-1}}\Vert  v\Vert^{\frac{3r-4}{2}}_{\mathbb{H}^1}.
\end{align*}
\end{lemma}
\begin{proof}
We recall Young's inequality
\begin{align*}
ab\leq \frac{q-1}{q}a^{\frac{q}{q-1}}+\frac{b^q}{q},\quad a,b>0,\; q>1. 
\end{align*}
For $2<r<3$,  applying the preceding estimate with $q=4-r$ leads to
\begin{align*}
\widetilde{C}\vert v\vert^3= \widetilde{C}\epsilon^{\alpha\frac{3-r}{4-r}}(\vert v\vert^4)^{\frac{3-r}{4-r}}\epsilon^{-\alpha\frac{3-r}{4-r}}\vert v\vert^{\frac{r}{4-r}}\leq \epsilon^{\alpha}\vert v\vert^4+\frac{\widetilde{C}^{4-r}}{4-r}\epsilon^{-\alpha(3-r)}\vert v\vert^r.
\end{align*}
Integrating the above estimate  over $\mathcal{D}$ leads to
\begin{align}
\label{cubic2}
\widetilde{C}\Vert v\Vert^3_{\mathbb{L}^3}\leq \epsilon^{\alpha}\Vert v\Vert^4_{\mathbb{L}^4}+\frac{\widetilde{C}^{4-r}}{4-r}\epsilon^{-\alpha(3-r)}\Vert v\Vert^r_{\mathbb{L}^r}.
\end{align}
 Let us recall the following  interpolation inequality (see \cite[Proposition 6.10]{Folland})
\begin{align*}
\Vert u\Vert_{\mathbb{L}^{q'}}\leq \Vert u\Vert^{\lambda}_{\mathbb{L}^{p'}}\Vert u\Vert^{1-\lambda}_{\mathbb{L}^{r'}},\quad u\in \mathbb{L}^{r'},\quad p'<q'<r', \quad \lambda=\frac{p'}{q'}\frac{r'-q'}{r'-p'}.
\end{align*}
Using the preceding interpolation inequality  with $p'=2$, $q'=r$ and $r'=4$, yields
\begin{align}
\label{cubic3}
\Vert v\Vert^r_{\mathbb{L}^r}\leq \Vert v\Vert^{4-r}_{\mathbb{L}^2}\Vert v\Vert^{2r-4}_{\mathbb{L}^4}\leq C_{\mathcal{D}}\Vert v\Vert^{4-r}_{\mathbb{L}^2}\Vert v\Vert^{2r-4}_{\mathbb{H}^{1}},
\end{align}
where in the last step we used the embedding $\mathbb{H}^{1}\hookrightarrow \mathbb{L}^4$. Using the interpolation inequality $\Vert v\Vert \leq \Vert v\Vert^{\frac{1}{2}}_{\mathbb{H}^{-1}}\Vert  \nabla v\Vert^{\frac{1}{2}}$, it follows from \eqref{cubic3} that
\begin{align*}
\Vert v\Vert^r_{\mathbb{L}^r}\leq C_{\mathcal{D}}\Vert v\Vert^{\frac{4-r}{2}}_{\mathbb{H}^{-1}}\Vert  v\Vert^{\frac{3r-4}{2}}_{\mathbb{H}^1}.
\end{align*}
Substituting the preceding inequality into \eqref{cubic2}  completes the proof of the lemma. 
\end{proof}

Below  we let $r$ in \lemref{Fundamentallemma}  be such that $2<r\leq \frac{8}{3}$, then it holds $\frac{1}{2}<\frac{3r-4}{4}\leq 1$. 
\begin{theorem}
 \label{mainresult1}
% Let \assref{assumption1} be fulfilled. 
 Let $u^{\epsilon}_{\mathsf{A}}$ be the solution of \eqref{model5} with large enough $K$ and let $u^{\epsilon}$ be the solution of \eqref{sch} with initial value $u^{\epsilon}_{\mathsf{A}}(0)=u^{\epsilon}(0)=u^{\epsilon}_0\in \mathbb{H}^{-1}$.
 For sufficiently small   $\eps > 0$, $\delta>0$, $\alpha>0$, $\eta\geq 0$ and any $2<r\leq \frac{8}{3}$, $\sigma,\gamma>0$ that satisfy
 \begin{align*}
 \gamma> \frac{(7-2\alpha)r+6\alpha-8}{r-2},\quad
 \sigma>\frac{3}{4}\gamma+\frac{1}{4}+2\delta+2\eta,
 \end{align*}
 there exist  positive constants $C$ and $C_{\delta,\eta}$ independent of $\epsilon$,  such that the following hold
 \begin{align}
 \label{result1}
  \mathbb{P}\left(\left\{\Vert u^{\epsilon}-u^{\epsilon}_{\mathsf{A}}\Vert^2_{L^{\infty}(0, T; \mathbb{H}^{-1})}\leq C\epsilon^{\gamma  -1}\right\}\right)\geq& 1-C_{\delta,\eta}\epsilon^{\delta+\eta},\\
  \label{result2}
 \mathbb{P}\left(\left\{\Vert w^{\epsilon}-w^{\epsilon}_{\mathsf{A}}\Vert^2_{L^1(0, T; \mathbb{H}^{-2})}\leq C\epsilon^{\frac{\gamma}{3}-1}\right\}\right)\geq& 1-C_{\delta,\eta}\epsilon^{\delta+\eta},\\
 \label{result3}
 \mathbb{P}\left(\left\{\Vert u^{\epsilon}-u^{\epsilon}_{\mathsf{A}}\Vert_{L^3(\mathcal{D}_T)}\leq C\epsilon^{\frac{\gamma}{3}}\right\}\right)\geq& 1-C_{\delta,\eta}\epsilon^{\delta+\eta},\\
 \label{result4}
 \mathbb{P}\left(\left\{\Vert u^{\epsilon}-u^{\epsilon}_{\mathsf{A}}\Vert_{L^4(\mathcal{D}_T)}\leq C\epsilon^{\frac{\gamma}{4}}\right\}\right)\geq& 1-C_{\delta,\eta}\epsilon^{\delta+\eta}.
  \end{align}
  Moreover, for any $\vartheta\in(0, 2-\frac{d}{2}]$ and $\kappa\geq 0$, there exists a constant $C_{\vartheta,\kappa}>0$ independent of $\epsilon$, such that
  \begin{align}
  \label{result5}
  \mathbb{P}\left(\left\{\Vert u^{\epsilon}-u^{\epsilon}_{\mathsf{A}}\Vert^2_{L^2(0, T; \mathbb{H}^{2-\frac{d}{2}-\vartheta})}\leq C\epsilon^{\frac{\gamma}{3}  -4}\right\}\right)\geq 1-C_{\delta,\eta}\epsilon^{\delta+\eta}-C_{\vartheta,\kappa}\epsilon^{\vartheta+\kappa}.
  \end{align}
 \end{theorem}
 %%%%%%%%%%%%%%%%%%%%%%%%%%%%
 \begin{remark}
 \label{NoiseStrength}
 Since $\alpha$ can be arbitrarily small, the smallest possible value for $\gamma$ in space dimension $d=2,3$ is $\gamma>  16$. Since $\delta$ and $\eta$ can be arbitrarily small, the smallest choice for the noise scaling parameter is $\sigma>\frac{49}{4}$.  
 \end{remark}
\begin{remark}
  \thmref{mainresult1} remains true in the case of trace class noise. It improves \cite[Theorem 3.10]{abk18} in the case $f(u)=u^3-u$ and $d=3$, since the obtained error estimates 
are in $L^p(\mathcal{D}_T)$, $p\in (2, 4]$, i.e., with $p$ exceeding the barrier $p\leq \frac{2d+8}{d+2}$ prescribed in \cite{abk18}. 
Note that, for sufficiently regular trace class noise, it is possible to derive an analogue of \eqref{result5}  in the the $\mathbb{H}^1$-norm, cf. \cite{abk18}. 
So far an analogous estimate for the stochastic Cahn-Hilliard equation with the space-time white noise
was missing, cf. \cite{BYZ22} for the case $d=2$.
Estimate \eqref{result5} depends on the spatial dimension: for $d=2$ it  holds up to $\mathbb{H}^1$-norm, while for $d=3$ it holds up to $\mathbb{H}^{\frac{1}{2}}$ (excluding the respective borderline cases).
This underlines the low regularity of the space-time white noise.
Assuming slightly better regularity of the noise (which is still lower than in \cite[Assumption 3.1]{abk18}), we achieve an estimate in $\mathbb{H}^1$ even in dimension $d=3$, see \secref{sec_h1} below.
\end{remark}

\begin{proof}
%In this subsection, we prove that $T_{\epsilon}(\omega)=T$ for all $\omega\in \Omega_{\delta,\eta, \epsilon}$. 
For $a, b\in\mathbb{R}$, let $a\wedge b:=\min\{a, b\}$. 
We set 
\begin{align*}
\gamma_1:=&(\gamma-1)\wedge (\frac{4}{3}(\sigma^*-2\delta-2\eta)-1)\wedge (2(\sigma^*-2\delta-2\eta)-1)\wedge (\frac{8}{3}(\sigma^*-2\delta)-1)\nonumber\\
&\wedge (4(\sigma^*-2\delta-2\eta)-1)\nonumber\\
=&(\gamma-1)\wedge (\frac{4}{3}(\sigma^*-2\delta-2\eta)-1).
\end{align*}
 Recall that $\sigma^*=\sigma-\frac{1}{4}$. Since $\delta$ and $\eta$ can be arbitrarily small, $\sigma^*>2\delta+2\eta$. 

We   aim to show that $T_{\epsilon}(\omega) = T$ for $\omega\in \Omega_{\delta, \eta,\epsilon}$. We proceed by contradiction and assume that $T_{\epsilon}  < T$ on $\Omega_{\delta, \eta,\epsilon}$. 
We consider $K$ large enough so that $\gamma_1\leq \frac{3K-5}{2}$. Using \lemref{Fundamentallemma} with $\alpha > 0$, H\"{o}lder's inequality and  \lemref{lemmaerrorOmega2}, yields for all $t\leq T_{\epsilon}$ 
\begin{align}
\label{Global1}
\int_0^t\Vert Y^{\epsilon}(s)\Vert^3_{\mathbb{L}^3}ds&\leq \epsilon^{\alpha}\int_0^t\Vert Y^{\epsilon}(s)\Vert^4_{\mathbb{L}^4}ds+ C\epsilon^{r\alpha-3\alpha}\int_0^t\Vert Y^{\epsilon}(s)\Vert^{\frac{4-r}{2}}_{\mathbb{H}^{-1}}\Vert  Y^{\epsilon}(s)\Vert^{\frac{3r-4}{2}}_{\mathbb{H}^1}ds\nonumber\\
&\leq C\epsilon^{\gamma_1+ 1+\alpha}+ C\epsilon^{r\alpha-3\alpha}\left(\sup_{s\in[0, t]}\Vert Y^{\epsilon}(s)\Vert^{2}_{\mathbb{H}^{-1}}\right)^{\frac{4-r}{4}}\int_0^t\Vert Y^{\epsilon}(s)\Vert^{\frac{3r-4}{2}}_{\mathbb{H}^1}ds\nonumber\\
 &\leq  C\epsilon^{\gamma_1+ 1+\alpha}+C \epsilon^{r\alpha-3\alpha}\epsilon^{\left(\frac{4-r}{4}\right)\gamma_1}\left(\int_0^t\Vert Y^{\epsilon}(s)\Vert^2_{\mathbb{H}^1}ds\right)^{\frac{3r-4}{4}}\\
 &\leq C\epsilon^{\gamma_1+ 1+\alpha}+C  \epsilon^{r\alpha-3\alpha}\epsilon^{\left(\frac{4-r}{4}\right)\gamma_1}\epsilon^{ 4-3r}\epsilon^{\left(\frac{3r-4}{4}\right)\gamma_1}\nonumber\\
 &=C\epsilon^{\gamma_1+ 1+\alpha}+C\epsilon^{ \frac{r}{2}\gamma_1+(\alpha-3)r+4-3\alpha}.\nonumber
\end{align}
The right hand side of the above inequality is bounded by $\epsilon^{\gamma}$ 
for sufficiently small $\eps$ if $\gamma_1+ 1+\alpha> \gamma$ and $ \frac{r}{2}\gamma_1+(\alpha-3)r+4-3\alpha> \gamma$. 

For $\frac{4}{3}(\sigma^*-2\delta-2\eta)> \gamma$ (i.e., for sufficiently large $\sigma> \frac{3}{4}\gamma+\frac{1}{4}+2\delta+2\eta$) we have $\gamma_1=\gamma-1$ 
and the requirement $ \frac{r}{2}\gamma_1+(\alpha-3)r+4-3\alpha> \gamma$  is equivalent to $\gamma>  \frac{(7-2\alpha)r+6\alpha-8}{r-2}$.
Consequently $\int_0^t\Vert Y^{\epsilon}(s)\Vert^3_{\mathbb{L}^3}ds\leq \epsilon^{\gamma}$ for all $t\leq T_{\epsilon}$ and $\omega\in\Omega_{\delta,\eta,\epsilon}$, which contradicts the definition of $T_{\epsilon}$. 
Hence, it holds that $T_{\epsilon}\equiv T$ on $\Omega_{\delta, \eta,\epsilon}$.

Recalling $R^{\epsilon}=Y^{\epsilon}+Z^{\epsilon}$ and noting that $T_{\epsilon}= T$ on $\Omega_{\delta, \eta,\epsilon}$, we deduce from Lemma \ref{LuboLemma1} and \ref{lemmaerrorOmega2} by the embedding $C(\mathcal{D})\hookrightarrow \mathbb{H}^{-1}$ that  on $\Omega_{\delta,\eta,\epsilon}$ it holds
\begin{align*}
\sup_{t\in[0, T]}\Vert R^{\epsilon}(t)\Vert^2_{\mathbb{H}^{-1}}&\leq 2\sup_{t\in[0, T]}\Vert Y^{\epsilon}(t)\Vert^2_{\mathbb{H}^{-1}}+2C\sup_{t\in[0, T]}\Vert Z^{\epsilon}(t)\Vert^2_{C(\mathcal{D})}\nonumber\\
&\leq C\epsilon^{\gamma_1}+C\epsilon^{2\sigma^*-4\delta-4\eta}\leq C\epsilon^{\gamma {-1}}+C\epsilon^{2\sigma^*-4\delta-4\eta}\leq C\epsilon^{\gamma { -1}},
\end{align*}
for $\gamma>  \frac{(7-2\alpha)r+6\alpha-8}{r-2}$ and $\sigma  >  \frac{3}{4}\gamma+\frac{1}{4}+2\delta+2\eta$ (recall $\sigma^*=\sigma - \frac{1}{4}$). 
Hence it follows that $\Omega_{\delta,\eta,\epsilon}\subseteq\{\omega\in \Omega: \Vert R^{\epsilon}\Vert^2_{L^{\infty}(0, T; \mathbb{H}^{-1})}\leq C\epsilon^{\gamma { -1}}\}$. Consequently, \lemref{LuboLemma1} yields that
\begin{align*}
\mathbb{P}\left(\left\{\Vert R^{\epsilon}\Vert^2_{L^{\infty}(0, T; \mathbb{H}^{-1})}\leq C\epsilon^{\gamma { -1}}\right\}\right)\geq \mathbb{P}(\Omega_{\delta,\eta,\epsilon})\geq 1-C\epsilon^{\delta+\eta}.
\end{align*}
This proves \eqref{result1}  since $R^{\epsilon} = u^\eps -u_{\mathsf{A}}^\eps$.

 Recalling that $R^{\epsilon}=Y^{\epsilon}+Z^{\epsilon}$, it follows from \eqref{Global1} and \lemref{LuboLemma1} that   on $\Omega_{\delta,\eta,\epsilon}$ we have
\begin{align}
\label{Cubic1}
\Vert R^{\epsilon}\Vert_{L^3(0, T; \mathbb{L}^3)}\leq C\epsilon^{\frac{\gamma}{3}}+C\epsilon^{\sigma^*-2\delta-2\eta}\leq C\epsilon^{\frac{\gamma}{3}},
\end{align}
for any $\gamma> \frac{5r-4}{2(r-2)}$ and $\sigma\geq \frac{\gamma}{3}+2\delta+2\eta+\frac{1}{4}$.  This implies that for such $\sigma$ and $\gamma$ we have $\Omega_{\delta,\eta,\epsilon}\subseteq\{\omega\in\Omega: \Vert R^{\epsilon}\Vert_{ L^3(0, T; \mathbb{L}^3)}\leq C\epsilon^{\frac{\gamma}{3}}\}$. Consequently form \lemref{LuboLemma1} we deduce
\begin{align*}
\mathbb{P}\left(\left\{\Vert R^{\epsilon}\Vert_{L^3(0, T; \mathbb{L}^3)}\leq C\epsilon^{\frac{\gamma}{3}}\right\}\right)\geq \mathbb{P}(\Omega_{\delta,\eta,\epsilon})\geq 1-C\epsilon^{\delta+\eta},
\end{align*} 
which yields \eqref{result3}. 

From Lemma \ref{lemmaerrorOmega2} and \ref{LuboLemma1} we deduce for $\gamma> \frac{(7-2\alpha)r+6\alpha-8}{r-2}$ and $\sigma  >  \frac{3}{4}\gamma+\frac{1}{4}+2\delta+2\eta$ that  on $\Omega_{\delta,\eta,\epsilon}$ it holds
\begin{align*}
\Vert R^{\epsilon}\Vert_{L^4(0, T; \mathbb{L}^4)}\leq C\epsilon^{\frac{\gamma_1+1}{4}}+C\epsilon^{\sigma^*-2\delta-2\eta}\leq C\epsilon^{\frac{\gamma}{4}}+C\epsilon^{\sigma^*-2\delta-2\eta}\leq  C\epsilon^{\frac{\gamma}{4}}.
\end{align*}
This implies  $\Omega_{\delta,\eta,\epsilon}\subseteq\{\omega\in\Omega:\, \Vert R^{\epsilon}\Vert_{L^4(0, T; \mathbb{L}^4)}\leq C\epsilon^{\frac{\gamma}{4}}\}$. Hence, using \lemref{LuboLemma1}, yields 
\begin{align*}
\mathbb{P}\left(\left\{\Vert R^{\epsilon}\Vert_{L^4(0, T; \mathbb{L}^4)}\leq C\epsilon^{\frac{\gamma}{4}}\right\}\right)\geq \mathbb{P}(\Omega_{\delta,\eta,\epsilon})\geq 1-C\epsilon^{\delta+\eta}.
\end{align*}
This completes the proof of \eqref{result4}. 

Using  the embedding $\mathbb{L}^1\hookrightarrow \mathbb{H}^{-2}$, \eqref{identity1}, the uniform boundedness of $u^{\epsilon}_{\mathsf{A}}$ (cf. \eqref{Alakiresult1})  and \eqref{Cubic1},
we obtain on $\Omega_{\delta, \eta, \eps}$ that
\begin{align}
\label{double2}
\Vert f(u^{\epsilon}_{\mathsf{A}})-f(u^{\epsilon})\Vert_{L^1(0, T; \mathbb{H}^{-2})}&\leq C\Vert f(u^{\epsilon}_{\mathsf{A}})-f(u^{\epsilon})\Vert_{ L^1(0, T; \mathbb{L}^1)}\nonumber\\
&  \leq C\Vert R^{\epsilon}\Vert_{L^3(0, T; \mathbb{L}^3)}+C\Vert R^{\epsilon}\Vert^2_{L^3(0, T; \mathbb{L}^3)}+\Vert R^{\epsilon}\Vert^3_{L^3(0, T; \mathbb{L}^3)}\nonumber\\
&\leq C\epsilon^{\frac{\gamma}{3}}.
\end{align}
  Recalling that $Y^{\epsilon}=R^{\epsilon}-Z^{\epsilon}$, using  the embeddings $C(\mathcal{D})\hookrightarrow \mathbb{L}^2$,  $\mathbb{L}^3\hookrightarrow\mathbb{L}^2$, \eqref{Cubic1} and Lemma \ref{LuboLemma1} yields on $ \Omega_{\delta, \eta, \eps}$ 
\begin{align}
\label{double1}
\Vert \Delta(Y^{\epsilon}+Z^{\epsilon})\Vert_{L^1(0, T; \mathbb{H}^{-2})}&\leq C\Vert Y^{\epsilon}\Vert_{L^{ 1}(0, T; \mathbb{L}^2)}+\Vert Z^{\epsilon}\Vert_{C(\mathcal{D}_T)} \leq  C\Vert R^{\epsilon}\Vert_{L^1(0, T; \mathbb{L}^3)}+\Vert Z^{\epsilon}\Vert_{C(\mathcal{D}_T)}\nonumber\\
& \leq  C\Vert R^{\epsilon}\Vert_{L^3(0, T; \mathbb{L}^3)}+\Vert Z^{\epsilon}\Vert_{C(\mathcal{D}_T)}\leq C\epsilon^{\frac{\gamma}{3}}+C\epsilon^{\sigma^*-2\delta-2\eta}\leq C\epsilon^{\frac{\gamma}{3}}, 
\end{align}
Recalling that $w^{\epsilon}-w^{\epsilon}_{\mathsf{A}}=-\epsilon\Delta(u^{\epsilon}-u^{\epsilon}_{\mathsf{A}})+\frac{1}{\epsilon}(f(u^{\epsilon})-f(u^{\epsilon}_{\mathsf{A}}))$, $R^{\epsilon}=u^{\epsilon}-u_{\mathsf{A}}=Y^{\epsilon}+Z^{\epsilon}$,  using \eqref{double1}, \eqref{double2} and the fact that $0<\epsilon\leq 1$, it follows  that  on $\Omega_{\delta, \eta,\epsilon}$ we have
\begin{align*}
\Vert w^{\epsilon}_{\mathsf{A}}-w^{\epsilon}\Vert_{L^1(0, T; \mathbb{H}^{-2})}\leq C\epsilon^{\frac{\gamma}{3}}+C\epsilon^{\frac{\gamma}{3}-1}\leq C\epsilon^{\frac{\gamma}{3}-1}.
\end{align*}
Therefore $\Omega_{\delta, \eta,\epsilon}\subseteq\{\omega\in\Omega: \Vert w^{\epsilon}_{\mathsf{A}}-w^{\epsilon}\Vert_{L^1(0, T; \mathbb{H}^{-2})}\leq C \epsilon^{\frac{\gamma}{3}-1}\}$, Using \lemref{LuboLemma1}  then yields
\begin{align*}
\mathbb{P}\left(\left\{\Vert w^{\epsilon}_{\mathsf{A}}-w^{\epsilon}\Vert_{L^1(0, T; \mathbb{H}^{-2})}\leq C \epsilon^{\frac{\gamma}{3}-1}\right\}\right)\geq \mathbb{P}(\Omega_{\delta,\eta,\epsilon})\geq 1-C_{\delta,\eta}\epsilon^{\delta+\eta}. 
\end{align*}
This completes the proof of \eqref{result2}. 

Since for any $\vartheta>0$, $1-\frac{d}{4}-\frac{\vartheta}{2}\leq \frac{1}{2}$, it follows from Lemma \ref{lemmaerrorOmega2} and \ref{Semigroup} that 
\begin{align*}
\Vert R^{\epsilon}\Vert^2_{L^2(0, T; \mathbb{H}^{2-\frac{d}{2}-\vartheta})}&\leq \Vert Y^{\epsilon}\Vert^2_{L^2(0, T; \mathbb{H}^{2-\frac{d}{2}-\vartheta})}+\Vert Z^{\epsilon}\Vert^2_{L^2(0, T; \mathbb{H}^{2-\frac{d}{2}-\vartheta})}\nonumber\\
&\leq \Vert Y^{\epsilon}\Vert^2_{L^2(0, T; \mathbb{H}^{1})}+\Vert Z^{\epsilon}\Vert^2_{L^2(0, T; \mathbb{H}^{2-\frac{d}{2}-\vartheta})}\nonumber\\
&\leq  C\epsilon^{\gamma_1-4} +C\epsilon^{2\sigma+\frac{\vartheta}{4}-1}\leq C\epsilon^{\frac{\gamma}{3} -4}+C\epsilon^{2\sigma+\frac{\vartheta}{4}-1}\leq C\epsilon^{\frac{\gamma}{3}  -4}
\end{align*}
on $\Omega_{\delta,\eta,\epsilon}\cap \widetilde{\Omega}_{\vartheta,\kappa,\epsilon}$. This implies  
\begin{align*}
\Omega_{\delta,\eta,\epsilon}\cap\widetilde{\Omega}_{\vartheta,\kappa,\epsilon}\subseteq\left\{\omega\in\Omega: \Vert R^{\epsilon}\Vert^2_{L^2(0, T; \mathbb{H}^{2-\frac{d}{2}-\vartheta})}\leq C\epsilon^{\frac{\gamma}{3}  -4}\right\}.
\end{align*} 
Using Lemma \ref{LuboLemma1}, \ref{Semigroup} and the identity $\Omega_{\delta, \eta,\epsilon}=(\Omega_{\delta,\eta,\epsilon}\cap\widetilde{\Omega}_{\vartheta,\kappa,\epsilon})\cup(\Omega_{\delta,\eta,\epsilon}\cap\widetilde{\Omega}^c_{\vartheta,\kappa,\epsilon})$, implies
\begin{align*}
\mathbb{P}\left(\left\{\Vert R^{\epsilon}\Vert^2_{L^2(0, T; \mathbb{H}^{2-\frac{d}{2}-\vartheta})}\leq C\epsilon^{\frac{\gamma}{3}  -4}\right\}\right)
&\geq \mathbb{P}(\Omega_{\delta,\eta,\epsilon}\cap\widetilde{\Omega}_{\vartheta,\kappa,\epsilon})=\mathbb{P}(\Omega_{\delta,\eta,\epsilon})-\mathbb{P}(\Omega_{\delta,\eta,\epsilon}\cap\widetilde{\Omega}^c_{\vartheta,\kappa,\epsilon})\nonumber\\
&\geq \mathbb{P}(\Omega_{\delta,\eta,\epsilon})-\mathbb{P}(\widetilde{\Omega}_{\vartheta,\kappa,\epsilon}^c)\geq 1-C_{\delta,\eta}\epsilon^{\delta+\eta}-C_{\vartheta,\kappa}\epsilon^{\vartheta+\kappa},
\end{align*}
which yields \eqref{result5}. The proof of \thmref{mainresult1} is therefore completed. 
\end{proof}

The corollary below provides an estimate of the difference between the solutions of the stochastic and the deterministic Cahn-Hilliard equation
and implies convergence to the solution of the deterministic problem for $\eps\rightarrow 0$ if $\delta+\eta>1$ and $\vartheta+\kappa>1$.
\begin{corollary}
\label{mainresult3}
Let the assumptions of \thmref{mainresult1} be fulfilled. Then it holds that
\begin{align*}
&\mathbb{E}\left[\Vert u^{\epsilon}-u^{\epsilon}_{D}\Vert^2_{L^{\infty}(0, T; \mathbb{H}^{-1})}+\frac{1}{\epsilon}\Vert u^{\epsilon}-u^{\epsilon}_{D}\Vert^4_{L^4(0, T; \mathbb{L}^4)}+ \epsilon ^4\Vert u^{\epsilon}-u^{\epsilon}_{D}\Vert^2_{L^2(0, T; \mathbb{H}^{2-\frac{d}{2}-\vartheta})}\right]\nonumber\\
&\leq C\epsilon^{\gamma-1}+C_{\delta, \eta}\epsilon^{\frac{\delta+\eta-1}{2}}+C_{\vartheta, \kappa}\epsilon^{\frac{\vartheta+\kappa-1}{2}}.
\end{align*}
\end{corollary}

\begin{proof}
From \cite[Theorem 2.1]{abc94} or \cite[Theorem 4.1 \& Remark 4.6]{abc94},  $u^{\epsilon}_{\mathrm{A}}\in C^2(\overline{\mathcal{D}}_T)\cap \mathbb{L}^2_0$ and
\begin{align}
\label{addition0}
\Vert u^{\epsilon}_{\mathrm{A}}-u^{\epsilon}_{D}\Vert^2_{L^{\infty}(0, T; \mathbb{H}^{-1})}+\Vert \nabla(u^{\epsilon}_{\mathrm{A}}-u^{\epsilon}_{D})\Vert^2_{L^2(0, T; \mathbb{L}^2)}\leq C\epsilon^{2\sigma}. 
\end{align}
Testing \eqref{model3} with $(-\Delta)^{-1}u^{\epsilon}_D$, along the same lines as in the  proof of \eqref{Regularesti}, yields
\begin{align}
\label{addition1}
\Vert u^{\epsilon}_D\Vert^2_{L^{\infty}(0, T; \mathbb{H}^{-1})}+\frac{1}{\epsilon}\Vert u^{\epsilon}_D\Vert^4_{L^4(0, T; \mathbb{L}^4)}+\epsilon\Vert \nabla u^{\epsilon}_{D}\Vert^2_{L^{2}(0, T; \mathbb{L}^2)}\leq C\epsilon^{-1}.
\end{align}
From \cite[Theorem 2.3]{abc94}, we have $\Vert u^{\epsilon}_{\mathrm{A}}-u^{\epsilon}_{D}\Vert_{C^1(\mathcal{D}_T)}\leq C\epsilon$. 
Using triangle inequality, \eqref{addition1} and  \thmref{mainresult2} below, it follows that 
\begin{align}
\label{addition2}
\mathcal{E}(u^{\epsilon}_{\mathrm{A}}):=\Vert u^{\epsilon}_{\mathrm{A}}\Vert^2_{L^{\infty}(0, T; \mathbb{H}^{-1})}+\frac{1}{\epsilon}\Vert u^{\epsilon}_{\mathrm{A}}\Vert^4_{L^4(0, T; \mathbb{L}^4)}+ \epsilon^4\Vert \nabla u^{\epsilon}_{\mathrm{A}}\Vert^2_{L^{2}(0, T; \mathbb{L}^2)}\leq C\epsilon^{-1}.
\end{align}  
Next, we consider the following subpace $\widetilde{\Omega}_1\subset\Omega$
\begin{align*}
\widetilde{\Omega}_1&=\left\{\omega\in\Omega:\; \Vert u^{\epsilon}-u^{\epsilon}_{\mathrm{A}}\Vert^2_{L^{\infty}(0, T; \mathbb{H}^{-1})}+\frac{1}{\epsilon}\Vert u^{\epsilon}-u^{\epsilon}_{\mathrm{A}}\Vert^4_{L^4(0, T; \mathbb{L}^4)}\right.\nonumber\\
&\left.\hspace{1cm}+\epsilon^4\Vert(-\Delta)^{1-\frac{d}{4}-\frac{\vartheta}{2}}(u^{\epsilon}-u^{\epsilon}_{\mathrm{A}})\Vert^2_{L^2(0, T; \mathbb{L}^2)}\leq C\epsilon^{\gamma  -1}\right\}. 
\end{align*}
By \thmref{mainresult1}, it holds $\mathbb{P}[\widetilde{\Omega}_1^c]\leq C_{\delta, \eta}\epsilon^{\delta+\eta}+C_{\vartheta, \kappa}\epsilon^{\vartheta+\kappa}$. We set
\begin{align*}
{\tt Err}_{\mathrm{A}}:=& \Vert u^{\epsilon}-u^{\epsilon}_{\mathrm{A}}\Vert^2_{L^{\infty}(0, T; \mathbb{H}^{-1})}+\frac{1}{\epsilon}\Vert u^{\epsilon}-u^{\epsilon}_{\mathrm{A}}\Vert^4_{L^4(0, T; \mathbb{L}^4)}\nonumber\\
&+\epsilon^4\Vert(-\Delta)^{1-\frac{d}{4}-\frac{\vartheta}{2}}(u^{\epsilon}-u^{\epsilon}_{\mathrm{A}})\Vert^2_{L^2(0, T; \mathbb{L}^2)}.
\end{align*}
Using the Cauchy-Schwarz inequality,  triangle inequality,  $\mathcal{E}_4(u^{\epsilon}_{\mathrm{A}}) \leq C \mathcal{E}(u^{\epsilon}_{\mathrm{A}})^{2}$ 
(by the embedding $\mathbb{H}^1 \hookrightarrow \mathbb{H}^{2-\frac{d}{2}-\vartheta}$), \eqref{Regularesti} (with $p=4$) and \eqref{addition2}, it holds that
\begin{align}
\label{relect2}
\mathbb{E}[{\tt Err}_{\mathrm{A}}]&=\int_{\Omega}1\!\!1_{\widetilde{\Omega}_1}{\tt Err}_Ad\mathbb{P}(\omega)+\int_{\Omega}1\!\!1_{\widetilde{\Omega}^c_1}{\tt Err}_{\mathrm{A}} d\mathbb{P}(\omega)\nonumber\\
&\leq C\epsilon^{\gamma  -1}+C\left(\mathbb{P}[\widetilde{\Omega}^c_1]\right)^{\frac{1}{2}}\left(\mathcal{E}_4(u^{\epsilon})+\mathcal{E}(u^{\epsilon}_{\mathrm{A}})^{2}\right)^{\frac{1}{2}}\leq C\epsilon^{\gamma -1}+C_{\delta, \eta}\epsilon^{\frac{\delta+\eta-1}{2}}+C_{\vartheta, \kappa}\epsilon^{\frac{\vartheta+\kappa-1}{2}}. 
\end{align}
The statement of the theorem then follows from \eqref{relect2}, \eqref{addition0} and \thmref{mainresult2} 
by the triangle inequality.
\end{proof}

As a consequence of \thmref{mainresult1} we obtain $\mathbb{P}$-a.s. convergence of the solution $u^\eps$ of the stochastic Cahn-Hilliard equation to the solution of the deterministic Hele-Shaw problem 
(\ref{model4}) in the sense that  $ \{(t,x)\in\mathcal{D}_T:\,\, t\in(0,T),\,\,\lim_{\eps\rightarrow 0} u^\eps(t,x) \rightarrow \pm 1\}$ respectively converge to  the exterior and interior of the interface $\{\Gamma_t\}_{t\in(0,T)}$.
The proof of the result follows along the lines of \cite[Corollary 4.5]{BYZ22}.
 \begin{corollary}
 There exists a subsequence $\{\epsilon_k\}_{k\in\mathbb{N}}$ such that 
 \begin{align*}
 \lim_{k\rightarrow\infty}u^{\epsilon_k}=1-2\chi_{\mathcal{D}^{-}}\quad \text{in}\; L^p(0,T; \mathbb{L}^p) \text{ for } p\in (2,4],
 \end{align*}
 $\mathbb{P}$-a.s. on $\Omega$, where $\mathcal{D}^- := \{(t,x)\in\mathcal{D}_T;\,\, t\in(0,T),\,\,x \in  \mathcal{D}_t^{-}\}$  and $\mathcal{D}^-_t$ is the interior of $\Gamma_t$ in $\mathcal{D}$. 
 \end{corollary}

%%%%%%%%%%%%%%%%%%%%%%%%%%%%%%%%%%%%%%%
%%%%%%%%%%%%%%%%%%%%%%%%%%%%%%%%%%%%%%%
%%%%%%%%%%%%%%%%%%%%%%%%%%%%%%%%%%%%%%%
%%%%%%%%%%%%%%%%%%%%%%%%%%%%%%%%%%%%%%%
%%%%%%%%%%%%%%%%%%%%%%%%%%%%%%%%%%%%%%%
\section{Limiting case with $\mathbb{H}^1$ spatial regularity}
\label{sec_h1}

In this section we consider the stochastic Cahn-Hilliard equation with a slightly more regular noise
\begin{align}
\label{sch1} %\label{model1}
du^{\epsilon} & =\Delta\left(-\epsilon\Delta u^{\epsilon}+\frac{1}{\epsilon}f(u^{\epsilon})\right)dt+\epsilon^{\sigma}d\widetilde{W}(t)
\quad \text{in } \mathcal{D}_T,
\end{align}
with the noise of the form
 \begin{align}
 \label{ModifNoise}
 \widetilde{W}(t,x)=\sum_{i\in \mathbb{N}^d}q_ie_i(x)\beta_i(t)\quad x\in \mathcal{D},\; t\in[0, T],  
 \end{align}
 where $\{q_i\}_{i\in\mathbb{N}^d}$ are such that $q_i\approx \lambda_i^{\frac{1}{2}-\frac{d}{4}-\frac{\upsilon}{2}}$ and $\upsilon \in (0,1]$ can be arbitrarily small.  
 The noise $\widetilde{W}$ is more regular than the space-time white noise $W$ \eqref{SpaceTimeWhiteNoise}, nevertheless 
 since $1-\frac{d}{2}-\upsilon\geq -\frac{d}{2}$ the series \eqref{ModifNoise} do not converge  in $\mathbb{L}^2$.

 We define the operator $Q: \mathbb{L}^2\rightarrow\mathbb{L}^2$ as 
 \begin{align*}
 Qu=\sum_{i\in\mathbb{N}^d}q_i(u, e_i)e_i\quad \forall u\in \mathbb{L}^2.
 \end{align*}
 Noting \eqref{Fractionaire1} we deduce that
 \begin{align*}
 \mathrm{Tr}\left((-\Delta)^{-1}Q\right)=\sum_{i\in\mathbb{N}^d}\left((-\Delta)^{-1}Qe_i, e_i\right)=\sum_{i\in\mathbb{N}^d}q_i\left((-\Delta)^{-1}e_i, e_i\right)\approx \sum_{i\in\mathbb{N}^d}\lambda_i^{-\frac{1}{2}-\frac{d}{4}-\frac{\upsilon}{2}}.
 \end{align*}
 For $d=3$ the above  identity implies that $\mathrm{Tr}\left((-\Delta)^{-1}Q\right)=\infty$, see \eqref{convergencecondition}. Hence, the condition \cite[Assumption 3.1]{abk18} is not satisfied for (\ref{ModifNoise}) in the case $d=3$. 

Similarly to Section~\ref{sec_exist} we introduce the stochastic convolution
\begin{align*}
\widetilde{Z}^{\epsilon}(t):=\epsilon^{\sigma}\int_0^te^{-(t-s)\epsilon\Delta^2}d\widetilde{W}(s)=\epsilon^{\sigma}\sum_{i\in\mathbb{N}^d}q_i\int_0^te^{-\lambda_i^2(t-s)\epsilon}e_id\beta_i(s)\quad t\in[0, T]. 
\end{align*}
Analogously to Lemma \ref{convollemma} and \ref{bruitlemma1}, for $1 \leq p< \infty$ one can show the estimate
\begin{align}\label{reg_z}
\mathbb{E}\left[\sup_{t\in[0, T]}\Vert \widetilde{Z}^{\epsilon}(t)\Vert^p_{\mathbb{L}^p}\right]+\mathbb{E}\left[\sup_{t\in[0, T]}\Vert \nabla\widetilde{Z}^{\epsilon}(t)\Vert^p\right]\leq C(p)\epsilon^{\left(\sigma-\frac{1}{2}\right)}.
\end{align}

Owing to the better regularity  properties of the convolution (\ref{reg_z}) it is straightforward to modify the proof of \thmref{regularitytheorem}
to show that the solution of \eqref{sch1} has the following regularity properties
\begin{align*}
u^{\epsilon}\in L^{\infty}\left(\Omega; C([0, T]; \mathbb{H}^{-1})\right)\cap L^2\left(\Omega; L^2(0, T; \mathbb{H}^{1})\right)\cap L^4\left(\Omega; L^4(0, T; \mathbb{L}^4)\right),
\end{align*}  
Moreover, for any $p\geq 2$ the following estimate holds
\begin{align*}
  %\label{RegularestiModified}
 \widetilde{\mathcal{E}}_p(u^{\epsilon})&:= \mathbb{E}\left[\Vert u^{\epsilon}\Vert^p_{L^{\infty}(0, T; \mathbb{H}^{-1})}+\epsilon^{\frac{p}{2}}\Vert \nabla u^{\epsilon}\Vert^p_{L^2(0, T; \mathbb{L}^2)}+\frac{1}{\epsilon^{\frac{p}{2}}}\Vert u^{\epsilon}\Vert^{2p}_{L^4(0, T; \mathbb{L}^4)}\right]\nonumber\\
 &\leq C\left(\epsilon^{-\frac{p}{2}}+\epsilon^{\left(\sigma-\frac{1}{2}\right)p}+\epsilon^{\left(2\sigma-\frac{3}{2}\right)p}\right).
\end{align*}

The remaining results of Section~\ref{sec_stoch} hold true with the fractional Sobolev norm replaced by the $\mathbb{H}^1$ norm.
In particular \eqref{result5} improves to
\begin{align*}
%\label{result5b}
\mathbb{P}\left(\left\{\Vert u^{\epsilon}-u^{\epsilon}_{\mathsf{A}}\Vert^2_{L^2(0, T; \mathbb{H}^{1})}\leq C\epsilon^{\frac{\gamma}{3} -4}\right\}\right)\geq 1-C_{\delta,\eta}\epsilon^{\delta+\eta}-C_{\upsilon,\kappa}\epsilon^{\upsilon+\kappa},
\end{align*}
for any $\upsilon>0$. 
The above estimate generalizes \cite[Theorem 3.10]{abk18} in the case $d=3$ to the case of less regular noise. 

Finally, we also obtain an analogue of the estimate in \coref{mainresult3}:
\begin{align*}
&\mathbb{E}\left[\Vert u^{\epsilon}-u^{\epsilon}_{D}\Vert^2_{L^{\infty}(0, T; \mathbb{H}^{-1})}+\frac{1}{\epsilon}\Vert u^{\epsilon}-u^{\epsilon}_{D}\Vert^4_{L^4(0, T; \mathbb{L}^4)}+ \epsilon^4\Vert u^{\epsilon}-u^{\epsilon}_{D}\Vert^2_{L^2(0, T; \mathbb{H}^{1})}\right]\nonumber\\
&\leq C\epsilon^{\gamma-1}+C_{\delta, \eta}\epsilon^{\frac{\delta+\eta-1}{2}}+C_{\upsilon, \kappa}\epsilon^{\frac{\upsilon+\kappa-1}{2}}.
\end{align*}

%%%%%%%%%%%%%%%%%%%%%%%%%%%%%%%%%%%%%%%%%%%%%%%%%%%%%%%%%%
%%%%%%%%%%%%%%%%%%%%%%%%%%%%%%%%%%%%%%%%%%%%%%%%%%%%%%%%%%
%%%%%%%%%%%%%%%%%%%%%%%%%%%%%%%%%%%%%%%%%%%%%%%%%%%%%%%%%%
%%%%%%%%%%%%%%%%%%%%%%%%%%%%%%%%%%%%%%%%%%%%%%%%%%%%%%%%%%
\section{The deterministic problem}
\label{sec_det}

In this section we derive improved estimates for the sharp interface limit of the deterministic Cahn-Hilliard equation \eqref{model3}.

 \begin{theorem}
 \label{mainresult2}
% Let \assref{assumption1} be fulfilled. 
 Let $\epsilon\in(0, 1]$ be sufficiently small, $2<r\leq \frac{8}{3}$ and $u^{\epsilon}_D$ be the solution to the deterministic Cahn-Hilliard equation \eqref{model3}. Then for any   $\alpha>0$ (arbitrarily small), $\gamma>  \frac{(7-2\alpha)r+6\alpha-8}{r-2}$ and $K$ in \eqref{Alakiresult1} large enough such that $\frac{3}{2}(K-1)\geq \gamma$,   there exists a constant $C=C(r)>0$ independent of $\epsilon$, such that
 \begin{align*}
\sup_{t\in[0, T]}\Vert u^{\epsilon}_D-u^{\epsilon}_{\mathsf{A}}\Vert^2_{\mathbb{H}^{-1}}+\epsilon^4\Vert u^{\epsilon}_D-u^{\epsilon}_{\mathsf{A}}\Vert^2_{L^2(0, T; \mathbb{H}^{1})}+ \frac{13}{8\epsilon}\Vert u^{\epsilon}_D-u^{\epsilon}_{\mathsf{A}}\Vert_{L^4(0, T; \mathbb{L}^4)}\leq C\epsilon^{\gamma-1}.
 \end{align*} 
 \end{theorem}
 
 \begin{remark}
 \thmref{mainresult2} improves \cite[Theorem 2.1]{abc94} in the case $f(u)=u^3-u$ and $d=3$, in the sense that here we obtain error estimates in $L^p(0, T; \mathbb{L}^p)$, $p\in (2, 4]$, i.e., with $p$ exceeding the barrier $p\leq \frac{2d+8}{d+2}$ prescribed in \cite{abc94}. Note that even for $d=2$ the error estimates in \cite{abc94} are only in $L^3(0, T; \mathbb{L}^3)$, while we obtain an error estimates in $L^4(0, T; \mathbb{L}^4)$. 
 \end{remark}
 
\begin{proof}
The proof goes along the same lines as the one of \thmref{mainresult1}, we only sketch some details. Setting $R^{\epsilon}_D:=u^{\epsilon}_D-u^{\epsilon}_{\mathsf{A}}$, with $u^{\epsilon}_D$ being the solution to the  CHE \eqref{model3}. Then $R^{\epsilon}_D$ satisfies the following PDE
\begin{align*}
\frac{d}{dt}R^{\epsilon}_D=-\epsilon\Delta^2R^{\epsilon}_D+\frac{1}{\epsilon}\Delta(f(R^{\epsilon}_D+u^{\epsilon}_{\mathsf{A}})-f(u^{\epsilon}_{\mathsf{A}}))+\Delta r^{\epsilon}_{\mathsf{A}}, \quad R^{\epsilon}_D(0)=0.
\end{align*}
Testing the above equation with $(-\Delta)^{-1}R^{\epsilon}_D$ leads to
\begin{align*}
\frac{1}{2}\frac{d}{dt}\Vert R^{\epsilon}_D\Vert^2_{\mathbb{H}^{-1}}+\epsilon\Vert  \nabla R^{\epsilon}_D\Vert^2+\frac{1}{\epsilon}\left(f(R^{\epsilon}_D+u^{\epsilon}_{\mathsf{A}})-f(u^{\epsilon}_{\mathsf{A}}), R^{\epsilon}_D\right)+\left( r^{\epsilon}_{\mathsf{A}}, R^{\epsilon}_D\right)=0.
\end{align*}
As in the proof of {\lemref{Formulalemma}}, we obtain
\begin{align}
\label{deter1}
&\Vert R^{\epsilon}_D\Vert^2_{\mathbb{H}^{-1}}+ \epsilon^4\int_0^t\Vert\nabla R^{\epsilon}_D(s)\Vert^2ds+\frac{13}{8\epsilon}\int_0^t\Vert R^{\epsilon}_D(s)\Vert^4_{\mathbb{L}^4}ds\nonumber\\
&\leq \frac{C}{\epsilon}\int_0^t\Vert R^{\epsilon}_D(s)\Vert^3_{\mathbb{L}^3}ds+ C\epsilon^{\frac{1}{2}} \int_0^t\Vert r^{\epsilon}_{\mathsf{A}}(s)\Vert^{\frac{3}{2}}_{C(\mathcal{D})}ds.
\end{align}
We define
\begin{align*}
T_{\epsilon}=T\wedge\inf\left\{t\in[0, T]:\;\int_0^t\Vert R^{\epsilon}_D(s)\Vert^3_{\mathbb{L}^3}ds>\epsilon^{\gamma}\right\},
\end{align*}
for some $\gamma>0$, which will be specified later.  

From \eqref{deter1}, using \eqref{contconv} and noting the definition of $T_{\epsilon}$, it follows for  all $t\leq T_{\epsilon}$  that
\begin{align}
\label{deter2}
&\sup_{0\leq s\leq t}\Vert R^{\epsilon}_D(s)\Vert^2_{\mathbb{H}^{-1}}+ \epsilon^4\int_0^t\Vert \nabla R^{\epsilon}_D(s)\Vert^2ds+\frac{13}{8\epsilon}\int_0^t\Vert R^{\epsilon}_D(s)\Vert^4_{\mathbb{L}^4}ds\nonumber\\
&\leq C\epsilon^{\gamma-1}+C\epsilon^{\frac{3K-5}{2}}\leq C\epsilon^{\gamma-1},
\end{align}
where we choose $K$ large enough such that $\gamma-1\leq \frac{3K-5}{2}$.

Next, we show that $T_{\epsilon}=T$ for suitable $\gamma$. We proceed by contradiction and assume that $T_{\epsilon}<T$. Then for all $t\leq T_{\epsilon}$, using \lemref{Fundamentallemma} with  $\alpha>0$ and \eqref{deter2} yields
\begin{align}
\label{relect1}
\int_0^t\Vert R^{\epsilon}_D(s)\Vert^3_{\mathbb{L}^3}ds&\leq  \epsilon^{\alpha}\int_0^t\Vert R^{\epsilon}_D(s)\Vert^4_{\mathbb{L}^4}ds+ C\epsilon^{r\alpha-3\alpha}\int_0^t\Vert R^{\epsilon}_D(s)\Vert^{\frac{4-r}{2}}_{\mathbb{H}^{-1}}\Vert  R^{\epsilon}_D(s)\Vert^{\frac{3r-4}{2}}_{\mathbb{H}^1}ds\nonumber\\
&\leq C\epsilon^{\gamma +\alpha}+C\epsilon^{r\alpha-3\alpha}\sup_{s\in[0, t]}\Vert R^{\epsilon}_D(s)\Vert_{\mathbb{H}^{-1}}^{\frac{4-r}{2}}\int_0^t\Vert R^{\epsilon}_D(s)\Vert^{\frac{3r-4}{2}}_{\mathbb{H}^1}ds
%\\
% &\leq  C\epsilon^{\gamma}+C\epsilon^{\left(\frac{4-r}{4}\right)(\gamma-1)}\left(\int_0^t\Vert R^{\epsilon}_D(s)\Vert^2_{\mathbb{H}^1}ds\right)^{\frac{3r-4}{4}}\nonumber
\\
\nonumber &\leq C\epsilon^{\gamma +\alpha}+C \epsilon^{r\alpha-3\alpha}\epsilon^{\left(\frac{4-r}{4}\right)(\gamma-1)}\epsilon^{4-3r}\epsilon^{\left(\frac{3r-4}{4}\right)(\gamma -1)}\nonumber\\
&= C\epsilon^{\gamma +\alpha}+C\epsilon^{ \frac{r}{2}\gamma+\frac{(2\alpha-7)r}{2}+4-3\alpha}.\nonumber
\end{align}
The right hand side of \eqref{relect1} is bounded above by $\epsilon^{\gamma}$ for sufficiently small $\epsilon$, if we have $\frac{r}{2}\gamma+\frac{(2\alpha-7)r}{2}+4-3\alpha> \gamma$, i.e. if $\gamma> \frac{(7-2\alpha)r+6\alpha-8}{r-2}$. Hence for such value of $K$ and $\gamma$ we have $\int_0^t\Vert R^{\epsilon}_D(s)\Vert^3_{\mathbb{L}^3}ds\leq \epsilon^{\gamma}$, which contradicts the definition of $T_{\epsilon}$. Hence $T_{\epsilon}=T$. It then follows from \eqref{deter2} that 
\begin{align*}
\sup_{0\leq t\leq T}\Vert R^{\epsilon}_D(t)\Vert^2_{\mathbb{H}^{-1}}+ \epsilon^4\int_0^T\Vert\nabla R^{\epsilon}_D(t)\Vert^2dt+\frac{13}{18\epsilon}\int_0^T\Vert R^{\epsilon}_D(t)\Vert^4_{\mathbb{L}^4}dt\leq C\epsilon^{\gamma -1},
\end{align*}
 which completes the proof.
\end{proof}

\section*{Acknowledgement}
Funded by the Deutsche Forschungsgemeinschaft (DFG, German Research Foundation) -- Project-ID 317210226 -- SFB 1283.

\bibliographystyle{plain}
\bibliography{stochch_hd_20240112_accepted}

\begin{thebibliography}{10}

\bibitem{abc94}
N.~D. Alikakos, P.~W. Bates, and X.~Chen.
\newblock Convergence of the {C}ahn-{H}illiard equation to the {H}ele-{S}haw
  model.
\newblock {\em Arch. Rational Mech. Anal.}, 128(2):165--205, 1994.

\bibitem{Banas19}
D.~Antonopoulou, \v{L}. Ba\v{n}as, R.~N\"{u}rnberg, and A.~Prohl.
\newblock Numerical approximation of the stochastic {C}ahn-{H}illiard equation
  near the sharp interface limit.
\newblock {\em Numer. Math.}, 147(3):505--551, 2021.

\bibitem{Antonopoulou_2023}
D.~C. Antonopoulou.
\newblock Higher moments for the stochastic cahn–hilliard equation with
  multiplicative fourier noise.
\newblock {\em Nonlinearity}, 36(2):053--1081, 2023.

\bibitem{abk18}
D.~C. Antonopoulou, D.~Bl\"{o}mker, and G.~D. Karali.
\newblock The sharp interface limit for the stochastic {C}ahn-{H}illiard
  equation.
\newblock {\em Ann. Inst. Henri Poincar\'{e} Probab. Stat.}, 54(1):280--298,
  2018.

\bibitem{sch_aposter}
\v{L}. Ba\v{n}as and C.~Vieth.
\newblock Robust a posteriori estimates for the stochastic {C}ahn-{H}illiard
  equation.
\newblock {\em Math. Comput.}, 92:2025--2063, 2023.

\bibitem{BYZ22}
\v{L}. Ba\v{n}as, H.~Yang, and R.~Zhu.
\newblock Sharp interface limit of stochastic {C}ahn-{H}illiard equation with
  singular noise.
\newblock {\em Potential Anal.}, 59:497--518, 2023.

\bibitem{Cahn1}
J.~W. Cahn.
\newblock On spinodal decomposition.
\newblock {\em Acta Metall.}, 9(9):795--801, 1961.

\bibitem{Cahn2}
J.~W. Cahn and J.~E. Hilliard.
\newblock Free energy of a nonuniform system. {I}. interfacial free energy.
\newblock {\em J. Chem. Phys.}, 28(2):258--267, 1958.

\bibitem{Debussche1}
G.~Da~Prato and A.~Debussche.
\newblock Stochastic {C}ahn-{H}illiard equation.
\newblock {\em Nonlinear Anal.}, 26(2):241--263, 1996.

\bibitem{DaPratoZabczyk}
G.~Da~Prato and J.~Zabczyk.
\newblock {\em Stochastic equations in infinite dimensions}.
\newblock Encyclopedia of mathematics and its applications. Cambridge Univ.
  Press, Cambridge, 2. ed. edition, 2014.

\bibitem{Folland}
G.~B. Folland.
\newblock {\em Real analysis}.
\newblock Pure and Applied Mathematics (New York). John Wiley \& Sons, Inc.,
  New York, second edition, 1999.
\newblock Modern techniques and their applications, A Wiley-Interscience
  Publication.

\bibitem{GAMEIRO2005693}
M.~Gameiro, K.~Mischaikow, and T.~Wanner.
\newblock Evolution of pattern complexity in the {C}ahn-{H}illiard theory of
  phase separation.
\newblock {\em Acta Materialia}, 53(3):693--704, 2005.

\bibitem{yang2019}
H.~Yang and R.~Zhu.
\newblock Weak solutions to the sharp interface limit of stochastic
  {C}ahn-{H}illiard equations, 2019.
\newblock \url{https://arxiv.org/abs/1905.09182}.

\end{thebibliography}

\end{document}